\documentclass{article}
\usepackage{graphicx} 
\usepackage{amsfonts, mathrsfs, amsmath, amsthm}
\usepackage{enumitem}
\usepackage{appendix}
\usepackage[section]{algorithm}
\usepackage{algpseudocode}
\usepackage{mathtools}
\usepackage{diagbox}

\usepackage{todonotes}

\usepackage[parfill]{parskip} 
\usepackage{subcaption}
\usepackage{svg}
\usepackage{xcolor}
\newcommand\numberthis{\addtocounter{equation}{1}\tag{\theequation}}

\newtheorem{theorem}{Theorem}[section]
\newtheorem{lemma}[theorem]{Lemma}
\newtheorem{corollary}{Corollary}[section]

\newtheorem{assumption}{Assumption}[section]
\theoremstyle{definition}

\newtheorem{remark}{Remark}[section]

\newcommand{\C}{\Sigma}
\newcommand{\M}{\mathcal{M}}
\newcommand{\cE}{\mathcal{E}}

\newcommand{\R}{\mathbb{R}}
\newcommand{\E}{\mathbb{E}}
\newcommand{\diag}{\text{Diag}}

\newcommand{\la}{\langle}
\newcommand{\ra}{\rangle}
\newcommand{\trace}{\text{trace}}
\newcommand{\ls}{\text{LS}}

\usepackage[normalem]{ulem}

\newcommand*\de{\mathop{}\!\mathrm{d}}

\newcommand*\h{\mathop{}\!\mathrm{H}}

\newcommand*\samethanks[1][\value{footnote}]{\footnotemark[#1]}

\numberwithin{equation}{section}
\title{Kalman-Langevin dynamics : exponential convergence, particle approximation and numerical approximation}
\author{%
Axel Ringh\thanks{%
Department of Mathematical Sciences,
Chalmers University of Technology and University of Gothenburg. \texttt{axelri@chalmers.se}, \texttt{akashs@chalmers.se}.}
\and
Akash Sharma\samethanks
}
\date{}

\begin{document}

\maketitle
\begin{abstract}
Langevin dynamics has found a large number of applications in sampling, optimization and estimation. Preconditioning the gradient in the dynamics with the covariance --- an idea that originated in literature related to solving estimation and inverse problems using Kalman techniques --- results in a mean-field (McKean-Vlasov) SDE.
We demonstrate exponential convergence of the time marginal law of the mean-field SDE to the Gibbs measure with
non-Gaussian potentials. This extends previous results, obtained in the Gaussian setting, to a broader class of potential functions. 
We also establish uniform in time bounds on all moments and convergence in $p$-Wasserstein distance. Furthermore, we show convergence of a weak particle approximation, that avoids computing the square root of the empirical covariance matrix, to the mean-field limit. Finally, we prove that an explicit numerical scheme for approximating the particle dynamics converges, uniformly in number of particles, to its continuous-time limit, addressing non-global Lipschitzness in the measure.

   \medskip
       
   
   \medskip

\end{abstract}

\section{Introduction}  

\begin{figure}
    \centering
    \includegraphics[width=1\textwidth, height=0.3\textheight]{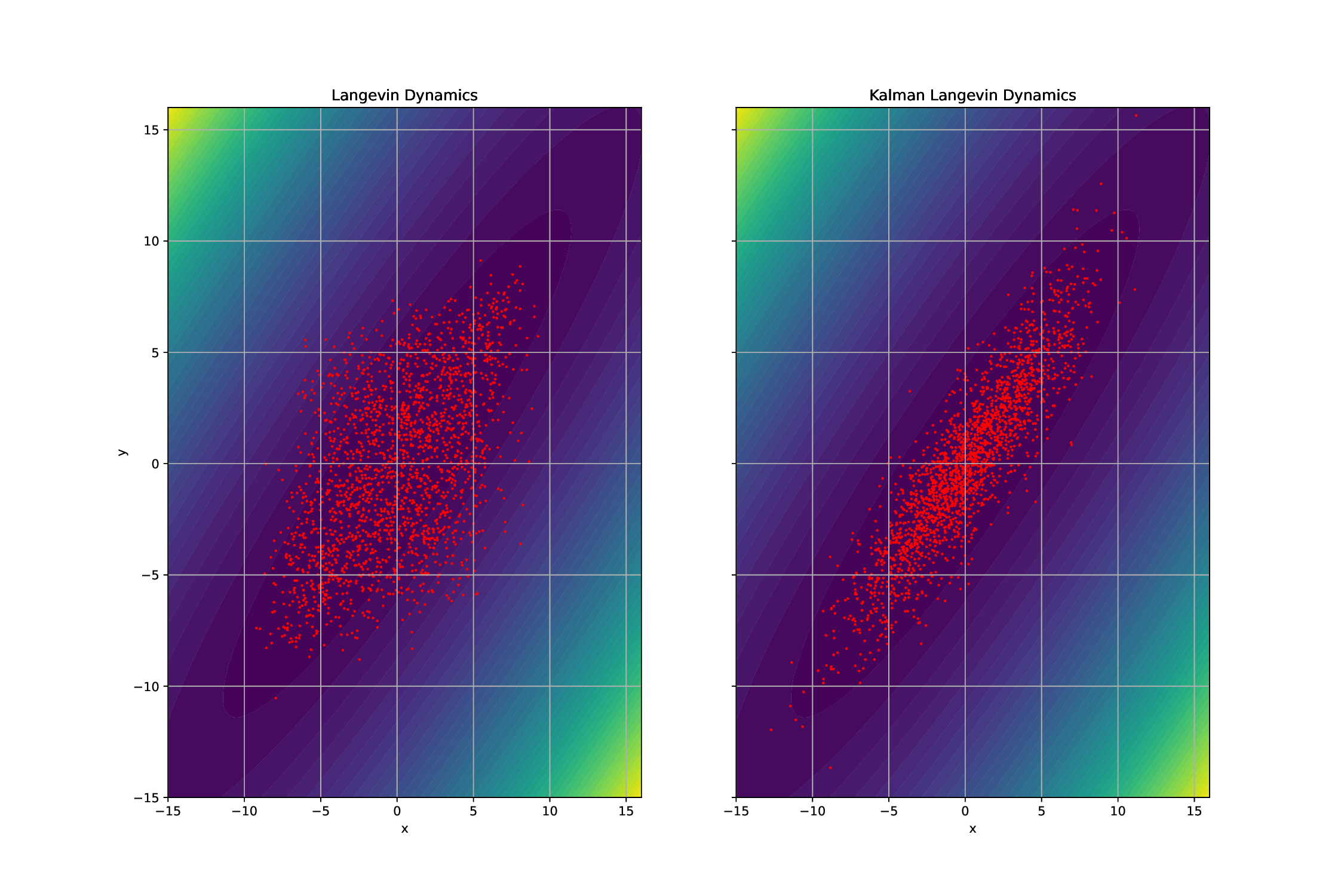}
    \caption{The difference in behavior of overdamped Langevin dynamics and Kalman-Langevin dynamics for potential function $U = 0.26 (x^2 + y^2) - 0.48 x y
$ at time $T = 1$ with $2000$ particles uniformly initialized in $[-15, 15]^2$ with $\beta = 1$.}
    \label{fig:introduction}
\end{figure}

Sampling and optimization techniques are, in many cases, the main ingredients to solutions of problems in applied mathematics and computational statistics. 
For instance, in molecular dynamics, sampling techniques focus on exploring the potential configurations of a molecule, while optimization comes into play when seeking the configuration with the minimum energy. Additionally, many optimization problems can be formulated as sampling problems within a Bayesian framework to account for uncertainty. 
In sampling, one looks for samples from the measure defined by 
\begin{align}\label{Gibb_measure}
    \mu(\de x) = \frac{1}{Z}e^{-\beta U(x)}\de x,
\end{align}
where $Z = \int_{\R^{d}} e^{-\beta U(x)}\de x$ is the normalizing constant, $\beta >0$ and $U : \mathbb{R}^d \rightarrow \R$. Connecting this to optimization, under relatively mild conditions on $U$ and for increasing $\beta$, the measure concentrates around global minima of $U$ \cite{hwang1980laplace, hasenpflug2024wasserstein}. 

For sampling purposes, in addition to traditional Markov chain Monte Carlo methods, overdamped Langevin dynamics, which finds its origin in statistical physics (see, e.g., \cite{rossky1978brownian}), is one popular choice of method. The overdamped Langevin dynamics (also known as Smoluchowski dynamics) driven by Brownian noise $W(t)$ is given by 
\begin{align}
     \de X(t) = -\nabla U(X(t)) \de t +  \sqrt{\frac{2}{\beta}} \de W(t), \quad X(0) \in \R^d,\label{eq:langevin}
\end{align}
and it leaves the Gibbs measure (\ref{Gibb_measure}) invariant under appropriate conditions on $U$.  
The dynamics in \eqref{eq:langevin} has also been used as an optimization method exploiting annealing techniques with a $\beta \to \infty$ (see \cite{chiang1987diffusion}), and as a starting point for developing new sampling methods.

On the other hand, the Kalman filter introduced in \cite{evensen1994sequential_kalman_filter} for state estimation has also found large number of applications in applied sciences (for example in oceanography \cite{evensen1996assimilation}, in reservoir modeling \cite{aanonsen2009ensemble}, and in weather forecasting \cite{houtekamer2001sequential}) for data assimilation and inverse problems. Inspired from the Kalman filter, an ensemble Kalman iterative procedure for solving inverse problems, called as ensemble Kalman inversion (EKI), was proposed in \cite{iglesias2013ensemble}, and the corresponding continuous-time limit, which is an interacting system of SDEs, was derived in \cite{schillings2017analysis_enkf}. The authors in \cite{blomker2019well} established well-posedness and convergence to the ground truth of the SDEs underlying EKI. The convergence to the mean-field limit was studied in \cite{ding2021ensemble_inversion} for the EKI model.  For the reflected EKI model, the well-posedness and convergence of particle system to mean-field limit in non-convex domain setting is established in \cite{hinds2024well}.
In this context, it is also worth noting other recent works on optimization and sampling using interacting particle systems demonstrating better capabilities to deal with the non-convexity and anisotropy of energy landscapes, with possible derivative free implementation, such as \cite{carrillo2018analytical, leimkuhler2018ensemble, Kovachki_stuart_2019, chada2020tikhonov, lindsey_weare2022ensemble_teleporting, molin2025controlled}.

Inspired from the
continuous-time limit of ensemble Kalman procedure with noise,  \cite{garbunoinigo2019interacting} proposed a covariance preconditioned overdamped Langevin dynamics which is a non-linear (in the sense of McKean) Markov process. This covariance preconditioned Langevin dynamics is driven by the following McKean-Vlasov SDEs:
\begin{subequations}\label{eq:ek_sdes}
\begin{align} \label{el_sde_mf}
     \de X(t) = -\Sigma(\mu_{t})\nabla U(X(t)) \de t +  \sqrt{\frac{2\Sigma(\mu_{t})}{\beta}} \de W(t),
\end{align}
with
\begin{align}
    \C(\mu_t)&= \int_{\R^{d}}(x - \M(\mu_t))(x - \M(\mu_t))^{\top} \de \mu_{t}(x), \\
     \M(\mu_t) &= \int_{\R^{d}}x \de \mu_{t}(x), 
\end{align}
\end{subequations}
where $W(t)$ is a $d$-dimensional Brownian motion and $\mu_{t} := \mathcal{L}^{X}_{t}$ is the time marginal law of $X$. This provides a new sampling method which portrays better performance in capturing  anisotropic energy landscapes, as illustrated in Figure~\ref{fig:introduction}. Furthermore, Kalman approximation of the gradient (see Remark~\ref{remark_3.1_deri_free}) provides derivative free technique for Laplace approximation of the targeted distribution. In \cite{garbunoinigo2019interacting}, provided that the initial distribution is not a Dirac distribution and that $U$ is a quadratic function of $x$, the authors showed the convergence of the law of (\ref{el_sde_mf}) to the Gibbs measure in relative entropy. 

In  \cite{carrillo_vaes_2021_cov_pre_fkp}, exponential convergence is obtained directly in Wasserstein distance for the linear setting, i.e., the setting when $U$ is quadratic and thus $\nabla U$ is linear.
One more step towards analysis is the convergence of interacting particle system to the mean field limit given by (\ref{eq:ek_sdes}), which is shown in 
\cite{ding2021ensemble_sampler} in the case of a quadratic potential. The propagation of chaos result is extended to a non-linear setting in \cite{vaes2024sharpchaos}, with optimal rate of convergence in terms of number of particles. 
This paper builds on the above, and the contributions of the paper are the following:
\begin{itemize}
    \item[(i)]  The first major contribution is that we establish exponential convergence of the law of the non-linear Langevin dynamics to the Gibbs measure for non-linear potential functions (in particular, for a quadratic potential with Lipschitz perturbation) in relative entropy (which also gives exponential convergence in $2-$Wasserstein distance). This significantly improves the linear setting results of \cite{garbunoinigo2019interacting,
carrillo_vaes_2021_cov_pre_fkp} as it covers a large class of Bayesian models with Gaussian priors. In addition, we also prove uniform in time $p$-th moment bound of the mean-field SDEs. Combining the results we straightforwardly have convergence in $p-$Wasserstein distance. One of the main ingredients of the analysis is matrix valued non-linear ordinary differential equations (ODEs). This approach bears resemblance to the approach based on Riccati type matrix valued ODEs which has been employed to obtain stability estimates in the case of ensemble Kalman(-Bucy) filters \cite{del_tugaut2018stability, del2023theoretical_kalman}. 
\item[(ii)] We prove the convergence of an interacting particle system to the mean-field SDEs (\ref{eq:ek_sdes}). In contrast to \cite{ding2021ensemble_sampler, vaes2024sharpchaos}, we study a weak particle approximation that avoids the need to compute the square root of the empirical covariance matrix. We refer to this as a weak approximation because it differs from the interacting particle systems proposed and studied in \cite{garbunoinigo2019interacting} and \cite{vaes2024sharpchaos} at the path level. To show this convergence, we follow the classical trilogy of arguments from \cite{sznitman1991topics} (also see \cite{graham_meleard_1996asymptotic}).
\item[(iii)]
The second major contribution is that we establish the uniform in $N$, where $N$ denotes number of particles, convergence of an implementable explicit numerical scheme, with fixed discretization time-step, to its continuous-time limit. In
\cite{blomker2018strongly}, the authors consider a one dimensional model SDE inspired from ensemble Kalman inversion and establish convergence of a numerical scheme to its continuous limit. For interacting particle system with non-global Lipschitzness in measure in $2-$Wasserstein metric, \cite{chen_goncalo_2024euler} propose a split-step scheme, however, the nonlinearity in terms of measure appear as a convolution which is not the case for \eqref{eq:ek_sdes}. 
\end{itemize}

The outline of the article is in line with the contributions mentioned above: in Section~\ref{sec:exp_conv}, we show the exponential convergence; in Section~\ref{sec:particle_approx}, we prove the propagation of chaos; and in Section~\ref{sec:time_disc}, we show the convergence of the numerical scheme.

\paragraph{Notation}
We will use the following notations in the paper. If $Q(t)$ is a time-varying symmetric positive semi-definite matrix, we denote $\lambda^{Q}_{\min}(t)$ and $\lambda^{Q}_{\max}(t)$ as smallest and largest eigenvalues, respectively, of $Q(t)$. For  square matrix $Q$, we denote its trace by $\trace(Q)$. If $O$ is a subset of $\R^d$ then $O^c$ represents its complement. For convenience, we denote $a \wedge b = \min(a,b)$ for $a, b \in \R$. We represent the space of probability measures on $\R^d$ as $\mathcal{P}(\R^d)$ and the subset of measures with bounded  $p$-moment  by $\mathcal{P}_p(\R^d)$. We denote with $C^{k}(\R^d)$ the space of $k$ times continuously differentiable functions, and with $C_{c}^{k}(\R^d)$ the corresponding functions with compact support. For a function $f : \mathbb{R}^d \rightarrow \mathbb{R}$, we denote its gradient vector and Hessian matrix by $\nabla f$ and $\nabla^2 f$, respectively. For $\mu \in \mathcal{P}(\R^d)$, we denote $\int_{\R^d} \phi(x) \mu(dx)$ as $\la \phi, \mu\ra $ provided that the integral is finite. Moreover, for $\mu, \nu \in \mathcal{P}(\R^d)$, by $\mu \ll \nu$ we mean that $\mu$ is absolutely continuous with respect to $\nu$. With $\la a\cdot b\ra$, we denote the scalar product between two vectors $a,b\in \R^d$.  For a $\R^{d \times m}$ matrix, we denote its Frobenius norm with $|\cdot|$, which reduces to the Euclidean norm for $m =1$.

We will use $\pi(x)$ to denote the density, with respect to the Lebesgue measure, of the Gibbs measure \eqref{Gibb_measure}, i.e., 
\begin{align}\label{eq:pi}
    \pi(x) =  \frac{1}{Z} e^{-\beta U(x)},
\end{align}
where $Z = \int_{\mathbb{R}^d}e^{- \beta U(x)} \de x$. Finally, we denote with $C$ a generic constant whose value may change from line to line.


\section{Exponential convergence of non-linear dynamics to equilibrium}\label{sec:exp_conv}
In this section, we prove exponential convergence of the non-linear Markov process driven by \eqref{eq:ek_sdes} to its invariant distribution (\ref{Gibb_measure}) in relative entropy. 
We do this by first showing a uniform in time lower bound on the smallest eigenvalue of $\C_{t}$, i.e., on $\lambda_{\min}^{\C}(t)$.
These results are all proved under the following assumption on the function $U$.

\begin{assumption} \label{first_assum_on_U}
    We assume that the potential function $U$ has the following form:
    \begin{align}
    U = \frac{1}{2} x^{\top} A x + V,
    \end{align}
    where $A $ is positive definite matrix with smallest eigenvalue  $\ell_A$ and largest eigenvalue $L_A$, and $V$ is $C^{1}(\R^{d})$ and Lipschitz continuous function with Lipchitz constant $L_V$.
\end{assumption}

\begin{remark}
Note that the above assumption allows for non-convexity. One example satisfying the above assumption is that of Rastrigin function, which is commonly used as both a benchmark objective function for optimization and as a test potential function in sampling problems.    
\end{remark}

We first recall the definition of relative entropy along with its relation to total variation distance and $2$-Wasserstein distance; for an in-depth treatment one can look at \cite{villani2003topics}.
For any two probability measures $\nu$ and $\mu$ on $(\R^d, \mathcal{B}(\R^d))$, the relative entropy of $\nu$ with respect to $\mu$ is defined as
\begin{equation*}
    \h(\nu \;| \;\mu) = 
\begin{cases}
  &  \int_{\R^d}  \ln \bigg(\frac{ \de \nu}{\de \mu}\bigg) \de \nu \quad\text{if} \quad \nu \ll \mu, 
  \\ 
  & +\infty \quad \text{else}.
\end{cases}
\end{equation*}
The $p$-Wasserstein distance on $\mathcal{P}_{p}(\R^d)$ is given by
\begin{align*}
    \mathcal{W}_{p}( \nu, \mu ) := \inf_{\gamma \in \Gamma(\nu, \mu)}\bigg\{ \big(\E | Y_{1} - Y_{2}|^{p}\big)^{1/p}, \; \text{Law}(Y_1, Y_2) = \gamma \bigg\}, 
\end{align*}
where infimum is taken over all couplings of $\nu$ and $\mu$, i.e., $\Gamma(\nu, \mu)$ is the set of all joint measures such that $\text{Law}(Y_1) = \nu$ and $\text{Law}(Y_2) = \mu$. 
We say that $\mu$ satisfies a log-Sobolev inequality with constant $\lambda_{\ls}$ if for all probability measures $\nu$ such that $\nu \ll \mu$ the following holds:
\begin{align}
    H(\nu \;|\;\mu) \leq \frac{1}{2\lambda_{\ls}} \int_{\mathbb{R}^d}\bigg| \nabla \ln \frac{\de \nu}{\de \mu}\bigg|^2\de \nu.
\end{align}

Next, we recall one inequality related to relative entropy. These inequalities implies convergence in total variation norm and $2-$Wasserstein distance if convergence in relative entropy is proven. 
If $\mu$ satisfies a log-Sobolev inequality
with constant $\lambda_{\ls}$, then, due to Talagrand's inequality, we have
\begin{align*}
    \mathcal{W}_{2}( \nu, \mu ) \leq \sqrt{\frac{2}{\lambda_{\ls}}  \h(\nu \;| \;\mu) }.
\end{align*}

The main contributions of this section are the following results.

\begin{theorem}\label{main_thrm_exp_conv}
    Let Assumption~\ref{first_assum_on_U} hold. Then the non-linear Langevin SDEs (\ref{eq:ek_sdes}) converges to its invariant measure exponentially in relative entropy, i.e., there exist positive constants $c_1$ and $c_2$ independent of $t$ such that
    \begin{align*}
        \h(\mathcal{L}^{X}_{t} \;|\; \pi) \leq c_1 e^{-c_2 t}.
    \end{align*}
\end{theorem}

\begin{remark}
An application of Talagrand's inequality implies that we also have exponential convergence in $2$-Wasserstein distance. This, in turn, ensures weak convergence of $\mathcal{L}^{X}_{t}$ to the Gibbs measure, as well as convergence of the second moment of $\mathcal{L}^{X}_{t}$ to the second moment of the Gibbs measure (see \cite[Theorem~7.12]{villani2003topics}).
\end{remark}

The proof of exponential convergence requires the following lemma. 
\begin{lemma} \label{min_eigen_bound_lemma}
    Under Assumption~\ref{first_assum_on_U}, the following bound holds:
    \begin{align}
        \lambda_{\min}^{\C}(t) \geq \min\bigg(\frac{1}{2\beta\big(   L_{A} + 2\beta L_{V}^{2}\big)}, \lambda^{\C}_{\min}(0)\bigg),
    \end{align}
    where $L_A$ and $L_V$ are from Assumption~\ref{first_assum_on_U}. 
\end{lemma}

Analogous to the uniform in time lower bound on the smallest eigenvalue of $\Sigma_t$, as stated in Lemma~\ref{min_eigen_bound_lemma} above, we can also prove uniform in time upper bound on the largest eigenvalue
(see Lemma~\ref{max_eigen_bound_lemma}). These uniform lower and upper bounds are of interest in their own right. Moreover, a consequence of the lower and upper bounds is uniform in time moment bounds of the mean-field dynamics which we mention here as a separate result and prove in Section~\ref{unif_mome_bound}.  

For brevity, we denote 
\begin{align}
  \ell_{\C} &:=  \min\bigg(\frac{1}{2\beta\big(   L_{A} + 2\beta L_{V}^{2}\big)}, \lambda^{\C}_{\min}(0)\bigg), \label{notation_l_sigma}\\ L_{\C} &:= \lambda^{\C}_{\max}(0)e^{-\ell_{A} t } + \frac{1}{\ell_{A}}\Bigg(\bigg(\frac{3}{4}\bigg)^{3}\bigg( \frac{ K_{V}}{ \ell_{A}}\bigg)^{2} + 
    \frac{3}{\beta^{2}}\Bigg) (1 - e^{-\ell_{A}t}). \label{notation_L_sigma}
\end{align}

\begin{theorem}\label{lem:X_bounded_moments}
    Under Assumption~\ref{first_assum_on_U}, the following bound holds for all $l >0 $:
       \begin{align}
    \E |X(t)|^{2 l} \leq C_{M}, \label{bound_mom_est}
\end{align}
 where $X(t)$ is from (\ref{eq:ek_sdes}) and $C_{M}$ is independent of $t$.
\end{theorem}
For brevity, we do not explicitly write the constant $C_{M}$ appearing (\ref{bound_mom_est}) but it can easily be inferred from the proof as we explicitly track all the constants.  

\begin{corollary}
    Let Assumption~\ref{first_assum_on_U} hold. Then, for all $p >0 $, the $p-$th moment of the law of mean-field SDEs (\ref{eq:ek_sdes}) converges to the $p-$th moment of the Gibbs measure (\ref{Gibb_measure}), i.e.,
    \begin{align}
        \int_{\R^d}|x|^p \mathcal{L}^{X}_t (\de x) \rightarrow  \int_{\R^d}|x|^p \mu(\de x)
    \end{align}
    as $t \rightarrow \infty$. 
\end{corollary}
\begin{proof}
 The proof directly follows from \cite[Theorem~2.20]{van2000asymptotic} using weak convergence of and uniform integrability of moments.
\end{proof}

\begin{corollary}
    Let Assumption~\ref{first_assum_on_U} hold. Then, for all $p >0 $, we have the following convergence:
    \begin{align}
        \mathcal{W}_p(\mathcal{L}^{X}_t,  \mu) \rightarrow 0
    \end{align}
    as $t \rightarrow \infty$. 
\end{corollary}
\begin{proof}
The result follows from  the previous corollary and \cite[Theorem~7.12]{villani2003topics}.
\end{proof}

\begin{lemma}\label{max_eigen_bound_lemma}
    Under Assumption~\ref{first_assum_on_U}, the following bound holds:
    \begin{align}
       \lambda^{\C}_{\max}(t) \leq \lambda^{\C}_{\max}(0)e^{-\ell_{A} t } + \frac{1}{\ell_{A}}\Bigg(\bigg(\frac{3}{4}\bigg)^{3}\bigg( \frac{ K_{V}}{ \ell_{A}}\bigg)^{2} + 
    \frac{3}{\beta^{2}}\Bigg) (1 - e^{-\ell_{A}t}),
    \end{align}
    where $L_A$ and $L_V$ are from Assumption~\ref{first_assum_on_U}.
\end{lemma}

 For the sake of convenience, in the following we will denote $
    \C_{t} := \C(\mu_{t}) $. 
We will represent $\C^{-1}_{t}$ as the inverse of $\C_{t}$. 
\subsection{Proof of Lemma~\ref{min_eigen_bound_lemma}}
\begin{proof}[Proof of Lemma~\ref{min_eigen_bound_lemma}]
We have the following due to (\ref{el_sde_mf}):
\begin{align*}
     \E X(t) = \E X(0) - \int_{0}^{t} \C_{s}\E\nabla  U(X(s)) \de s, \quad 
     \E X(t)^{\top} = \E X(0)^{\top} -  \int_{0}^{t} \E( \nabla U(X(s)))^{\top}\C_{s}^{\top} \de s.
\end{align*}
Consider the following stochastic dynamics:
\begin{align*}
    \de (X(t) - \E X(t)) =   -  \C_t [\nabla U(X(t)) - \E \nabla U(X(t))]\de t +  \sqrt{\frac{2\C_{t}}{\beta}}  \de W(t).
\end{align*}

    Using Ito's product rule, we have
\begin{align*}
    \de (X(t) - \E X(t)) (X(t) &- \E X(t))^{\top} = 
    - [ (X(t) - \E X(t)) (\nabla U(X(t)) - \E \nabla U(X(t)))^{\top}] \C_{t} \de t   
    \\  & \quad 
    - \C_{t}[ (\nabla U(X(t)) - \E \nabla U(X(t)))(X(t) - \E X(t))^{\top}]\de t + \frac{2}{\beta}\C_{t}\de t
    \\  & \quad 
          + \sqrt{\frac{2}{\beta}} (X(t) - \E X(t)) \de W(t)^{\top} \sqrt{\C_{t}} + \sqrt{\frac{2}{\beta}} \sqrt{\C_{t}}\de W(t) (X(t) - \E X(t))^{\top}.
 \end{align*}
    Therefore,
    \begin{align*}
        \de \C_{t} &= 
    - \E[ (X(t) - \E X(t)) (\nabla U(X(t)) - \E \nabla U(X(t)))^{\top}] \C_{t} \de t   
    \\  & \quad 
    - \C_{t}\E[ (\nabla U(X(t)) - \E \nabla U(X(t)))(X(t) - \E X(t))^{\top}]\de t + \frac{2}{\beta}\C_{t}\de t. \numberthis \label{n_langevin_eqn_2.3}
    \end{align*}

   Due to (\ref{n_langevin_eqn_2.3}), we have the following matrix valued differential equation:
   \begin{align*}
       \frac{\de }{\de t}\C_{t}^{-1} & = - \C_{t}^{-1} \frac{\de \C_{t}}{\de t} \C_{t}^{-1}
     =    \C_{t}^{-1}\E[ (X(t) - \E X(t)) (\nabla U(X(t)) - \E \nabla U(X(t)))^{\top}]    
    \\  & \quad 
    + \E[ (\nabla U(X(t)) - \E\nabla U(X(t)))(X(t) - \E X(t))^{\top}]\C_{t}^{-1} - \frac{2}{\beta}\C_{t}^{-1}. \numberthis \label{n_lange_eqn_2.6}
   \end{align*}
Using Assumption~\ref{first_assum_on_U}, we get
\begin{align*}
      \frac{\de }{\de t}\C_{t}^{-1}  & 
      = 2A + \C_{t}^{-1}\E[ (X(t) - \E X(t)) (\nabla V(X(t)) - \E \nabla V(X(t)))^{\top}]    
    \\  & \quad 
    + \E[ (\nabla V(X(t)) - \E  \nabla V(X(t)))(X(t) - \E X(t))^{\top}]\C_{t}^{-1} - \frac{2}{\beta}\C_{t}^{-1}. 
\end{align*}
For any $y \in \R^{d}$ with $|y| \neq 0$, we have
\begin{align*}
     \frac{\de}{\de t} y^\top\C_t^{-1} y & =  y^{\top}\frac{\de}{\de t}\C_t^{-1}y 
      = 2y^{\top} A y + y^\top\C^{-1}_{t}\E[ Z_t  P_t^{\top}] y  +  y^{\top}\E[  P_t Z_t^{\top} ] \C_{t}^{-1} y  - \frac{2}{\beta}y^{\top}\C_t^{-1}y
      \\  &
      = 2y^{\top} A y + 2y^\top\C^{-1}_{t}\E[ Z_t  P_t^{\top}] y    - \frac{2}{\beta}y^{\top}\C_{t}^{-1}y,  \numberthis \label{n_langevin_eqn:2.6}
\end{align*}
where, for brevity, we have denoted $Z_t := X(t) - \E X(t)$ and $ P_t := \nabla V(X(t)) - \E\nabla V(X(t))$. 

We first deal with the second term on the right hand side of (\ref{n_langevin_eqn:2.6}) by using Cauchy-Bunyakovsky-Schwarz inequality:
\begin{align*}
    y^\top\C^{-1}_{t}\E[ Z_t  P_t^{\top}] y &=  \E [ y^\top\C^{-1}_{t} Z_t  P_t^{\top} y]  
    \leq  (\E [ y^\top\C^{-1}_{t} Z_t]^{2})^{1/2} (\E[ P_t^{\top} y]^{2})^{1/2}
    \\  &
    \leq 2L_{V} |y| (\E [ y^\top\C^{-1}_{t} Z_t]^{2})^{1/2},
\end{align*}
where $L_V$ is as introduced in Assumption~\ref{first_assum_on_U}.

Noting that
\begin{align*}
    \E [ y^\top\C^{-1}_{t} Z_t]^{2} & = \E ([ y^\top\C^{-1}_{t} Z_t] [ Z_t^{\top}\C^{-1}_{t} y])
    = y^{\top}\C^{-1}_{t} \E (Z_t Z_t^{\top}) \C^{-1}_{t}y 
    = y^{\top}\C^{-1}_{t} y,
\end{align*}
since $ \E (Z_t Z_t^{\top}) = \C_t$. 
Therefore,
\begin{align}
     y^\top\C^{-1}_{t}\E[ Z_t  P_t^{\top}]y  \leq 2L_{V} |y| \sqrt{y^{\top} \C_{t}^{-1}y }.
\end{align}

This implies
\begin{align*}
     \frac{\de }{\de t} y^\top\C_t^{-1}y &= 2y^{\top} A y + 2y^\top\C^{-1}_{t}\E[ Z_t  P_t^{\top}] y    - \frac{2}{\beta}y^{\top}\C_{t}^{-1}y
        \\  & 
     \leq 2y^{\top} A y + 4L_{V} |y| \sqrt{y^{\top} \C_{t}^{-1}y }  - \frac{2}{\beta}y^{\top}\C_{t}^{-1}y. \numberthis \label{ildo_eqn_4.15}
\end{align*}
For the sake of convenience, we denote
\begin{align*}
    \Lambda(t, y) = \frac{y^{\top} \C^{-1}_{t} y}{|y|^{2}}.
\end{align*}
Dividing (\ref{ildo_eqn_4.15}) by $|y|^2$, with $|y| \neq 0$, on both sides, we ascertain
\begin{align*}
    \frac{\de}{\de t} \Lambda(t, y) \leq 2L_{A} +   4L_{V}  \sqrt{\Lambda(t, y) }  - \frac{2}{\beta}\Lambda(t, y),
\end{align*}
where $
    L_{A} := \sup_{y\;; \;|y| \neq 0} \frac{y^{\top} A y}{|y|^2}$. Using generalized Young's inequality ($ ab \leq  a^{2} / (2 \epsilon) + \epsilon b^{2}/2$, where $a, b, \epsilon >0$, with 
$a = \sqrt{\Lambda(s, y)}$, $b = L_{V}$, $\epsilon = 2\beta$), we have
\begin{align*}
 \frac{\de}{\de t} \Lambda(t, y) 
 & \leq 2L_{A} +   4\bigg( \beta L_{V}^{2} +  \frac{\Lambda(t, y) }{4\beta }\bigg)  - \frac{2}{\beta}\Lambda(t, y) 
 = 2L_{A} + 4\beta L^{2}_{V} - \frac{1}{\beta}\Lambda(t, y).
\end{align*}    
Using $e^{\frac{1}{\beta}t}$ as the integrating factor, we get
\begin{align*}
    \frac{\de}{\de t} e^{\frac{1}{\beta}t}\Lambda(t,y) \leq  \big(2L_{A} + 4\beta L_{V}^{2}\big) e^{\frac{1}{\beta}t}.   
\end{align*}
Therefor, 
we obtain
\begin{align}
     e^{\frac{1}{\beta}t}\Lambda(t,y)  \leq   \Lambda(0,y) + 2\beta \big(   L_{A} + 2\beta L_{V}^{2}\big)  \big( e^{\frac{1}{\beta}t} -1\big).
\end{align}
Hence, we finally have
\begin{align}
    \Lambda(t,y)  \leq  \Lambda(0,y)e^{-\frac{1}{\beta}t} + 2\beta \big(   L_{A} + 2\beta L_{V}^{2}\big)  \big(1 - e^{-\frac{1}{\beta}t}\big).
\end{align}
Taking supremum over $y \in \R^d$, with $|y| \neq 0$
and using the fact that $\C^{-1}_{t}$ is symmetric, we get
\begin{align*}
    \lambda^{\C^{-1}}_{\max}(t) \leq  \lambda^{\C^{-1}}_{\max}(0) e^{-\frac{1}{\beta}t}  + 2\beta\big(   L_{A} + 2\beta L_{V}^{2}\big) \big(1 -  e^{-\frac{1}{\beta}t  }  \big).
\end{align*}
Take $g(t) := \lambda^{\C^{-1}}_{\max}(0) e^{-\frac{1}{\beta}t}  + \beta\big(   2L_{A} + \beta L_{V}^{2}\big) \big(1 -  e^{-\frac{1}{\beta}t  }  \big)$, then $g'(t)$ is given by
\begin{align*}
    g'(t) =  - \frac{1}{\beta}\lambda^{\C^{-1}}_{\max}(0) e^{-\frac{1}{\beta}t}  + 2(   L_{A} + 2\beta L_{V}^{2}) e^{-\frac{1}{\beta}t}
    = \left( - \frac{1}{\beta}\lambda^{\C^{-1}}_{\max}(0) + 2(   L_{A} + 2\beta L_{V}^{2}) \right) e^{-\frac{1}{\beta}t}.
\end{align*}
This means that the sign of $g'$ is constant, suggesting that $g(t)$ is a monotonic function which either attend supremum at $t = 0$ or when $t \rightarrow \infty$. In that case, if 
\begin{align}
    2\beta(L_{A} + 2\beta L_{V}^{2}) \geq \lambda^{\C^{-1}}_{\max}(0),
\end{align}
then $
     \lambda^{\C^{-1}}_{\max}(t) \leq ´2\beta\big(   L_{A} + 2\beta L_{V}^{2}\big) $. However, if we choose initial distribution $\mu_0$ such that
\begin{align}
     2\beta(L_{A} + 2\beta L_{V}^{2}) \leq \lambda^{\C^{-1}}_{\max}(0),
\end{align}
then $
     \lambda^{\C^{-1}}_{\max}(t) \leq ´\lambda^{\C^{-1}}_{\max}(0)$. Therefore, we have
\begin{align*}
     \lambda^{\C^{-1}}_{\max}(t) \leq ´\max \bigg( 2\beta\big(   L_{A} + 2\beta L_{V}^{2}\big), \;\lambda^{\C^{-1}}_{\max}(0)\bigg).
\end{align*}
Note that $\lambda^{\C^{-1}}_{\max}(t)  =  1/\lambda^{\C}_{\min}(t) $. This implies 
\begin{align}
    \lambda^{\C}_{\min}(t) \geq ´\min\bigg(\frac{1}{2\beta\big(   L_{A} + 2\beta L_{V}^{2}\big)}, \lambda^{\C}_{\min}(0)\bigg).
\end{align}
\end{proof}

\subsection{Proof of Theorem~\ref{main_thrm_exp_conv}}

Before we proceed with the proof of Theorem~\ref{main_thrm_exp_conv}, we need the following result, which  
is direct consequence of \cite[Theorem~0.1]{cattiaux_guillin2022functional_supplement}. In order to state the result, we need the following notation: We denote the Poincar\'{e} constant for $\mu$ by $C_{P}(\mu)$.  
Let $\lambda_{\ls}(\mu_A)$ denote the log-Sobolev constant of $\mu_A(\de x) = e^{-x^{\top}Ax/2}\de x$. We denote $C_{\ls}(\mu_A) = 2/\lambda_{\ls}(\mu_A) $.

\begin{lemma}
    Let Assumption~\ref{first_assum_on_U} hold. Then, the measure $ \mu =  \frac{1}{Z}e^{-U(x)} \de x$ satisfies log-Sobolev inequality, i.e., for all $\nu \ll \mu $
    \begin{align*}
        \mathrm{H}(\nu \;| \;\mu) \leq \frac{1}{2\lambda_{\mathrm{LS}}}\int_{\R^d} \bigg|\nabla \ln\bigg(\frac{\de \nu}{\de \mu}\bigg)\bigg|^2 \de \nu, 
    \end{align*}
    where  log-Sobolev constant satisfies 
    \begin{align}
    \frac{1}{2 \lambda_{\mathrm{LS}}}\leq  \frac{1}{4}\big(K_1(1 + a_1^{-1})(1 + a_2^{-1})  + K_2\big((1 + a_2)(1+ a_1^{-1})/4 +  a_1^2 /2  \big)\big),
    \end{align}
    where $K_1 = C_{\mathrm{LS}}^A$, $K_2 = C_{\text{P}}(\mu) (2 + \la V, \mu_A\ra + L_U C_{\text{P}}(\mu)C_{\mathrm{LS}}(\mu_A)$, $a_1 = \frac{1}{K_2^{1/3}}\bigg( K_1 + \frac{K_2}{2}\bigg)^{2/3}$ and $ a_2 = 2 \sqrt{\frac{K_1}{K_2}} $.
\end{lemma}
\begin{proof}
First note that, thanks to Assumption~\ref{first_assum_on_U}, $\mu$ satisfies Poincar\'{e} inequality \cite{bakry2008simple}. 
Note that $\mu$ satisfies Poincare inequality with constant $C_{\text{P}}(\mu)$. Denote $C_{\ls}(\mu) = 2/\lambda_{\ls}$.
    From \cite[Theorem~0.1]{cattiaux_guillin2022functional_supplement}, we have for all $a_1>0$ and $a_2>0$
    \begin{align*}
        C_{\ls}(\mu) \leq  (1 + a_1^{-1})(1 + a_2^{-1}) C_{\ls}^A + C_{\text{P}}(\mu) (2 + \la V, \mu_A\ra + L_U C_{\text{P}}(\mu)C_{\ls}(\mu_A)\big((1 + a_2)(1+ a_1^{-1})/4 +  a_1^2 /2  \big).
    \end{align*}
Consider a function $f(a_1, a_2) = K_1(1 + a_1^{-1})(1 + a_2^{-1})  + K_2\big((1 + a_2)(1+ a_1^{-1})/4 +  a_1^2 /2  \big)$ for some $K_1, K_2 >0$, then $f $ attains its minimum at 
\begin{align*}
    a_1 = \frac{1}{K_2^{1/3}}\bigg( K_1 + \frac{K_2}{2}\bigg)^{2/3}, \quad
    a_2 = 2 \sqrt{\frac{K_1}{K_2}},
\end{align*}
which completes the proof.
\end{proof}

\begin{proof}[Proof of Theorem~\ref{main_thrm_exp_conv}]
Consider the Fokker-Plank equation governing the probability distribution of non-linear Langevin dynamics (\ref{el_sde_mf}) as
\begin{align}
    \frac{\partial \rho}{\partial t}(t,x) = \nabla\cdot (\rho(t,x)\C(\mu_{t})\nabla U(x)) + \frac{1}{\beta}\nabla \cdot (\C(\mu_{t}) \nabla\rho(t,x)).     
\end{align}
We can write potential function as $
    U(x) = -\frac{1}{\beta }\ln \pi(x) - \frac{1}{\beta }\ln Z
$, $Z = \int_{\R^d}e^{- \beta U(x) } dx$. This results in the following PDE in divergence form:
\begin{align}\label{eldo_eqn_5.3}
     \frac{\partial \rho}{\partial t}(t,x) &= -\frac{1}{\beta }\nabla\cdot \big(\rho(t,x)\C(\mu_{t})\nabla\ln(\pi(x))\big) + \frac{1}{\beta }\nabla \cdot  \big(\rho(t,x)\C(\mu_{t})\nabla \ln(\rho(t,x))\big) \nonumber 
     \\  & 
     = \frac{1}{\beta }\nabla \cdot \bigg( \rho(t,x) \C(\mu_{t})\nabla \ln\Big( \frac{\rho(t,x)}{\pi(x)}\Big)\bigg).     
\end{align}

We denote $\rho_{t} := \rho(t, \cdot)$. Consider the  relative Boltzmann-Shanon entropy (also known as KL-divergence) in terms of density $\rho$ with respect to $\pi$ :
\begin{align}
    \h(\rho_{t}\; | \; \pi) = \int_{\mathbb{R}^{d}} \rho_{t}\ln\bigg(\frac{\rho_{t}}{\pi}\bigg) \de x.
\end{align}
We will discuss below the behaviour of above functional along the solution of Fokker-Plank PDE (\ref{eldo_eqn_5.3}):
\begin{align*}
    \frac{\de }{\de t}\h(\rho_t
    \; | \;\pi) & = \int_{\mathbb{R}^{d}} \frac{\partial }{\partial t}\big(\rho_{t}(x) \ln (\rho_{t}(x))\big) - \frac{\partial }{\partial t}\big(\rho_{t} \ln (\pi(x))\big) \de x
    \\& =
    \int_{\mathbb{R}^{d}} \Big(\frac{\partial }{\partial t}\rho_{t}(x)  + \ln(\rho_{t}(x))  \frac{\partial }{\partial t}\rho_{t}(x)  - \ln(\pi(x)) \frac{\partial }{\partial t}\rho_{t}(x) \Big) \de x 
     = \int_{\mathbb{R}^{d}} \Big(\ln\bigg(\frac{\rho_{t}(x)}{\pi(x)}\bigg)\frac{\partial }{\partial t}\rho_{t}(x) \Big) \de x,
\end{align*}
where we have used the conservation of mass with time resulting in $\frac{\partial}{\partial t}\int_{\mathbb{R}^{d}}\rho_{t}(x) dx = 0$. 
Using (\ref{eldo_eqn_5.3}) and integration by parts, we obtain
\begin{align}
 \frac{\de }{\de t}\h(\rho_{t}\; | \;\pi) =    \int_{\mathbb{R}^{d}}  -\frac{1}{\beta}  \bigg( \nabla \ln\Big( \frac{\rho_{t}(x)}{\pi(x)}\Big) \cdot \C(\mu_{t})\nabla \ln\Big( \frac{\rho_{t}(x)}{\pi(x)}\Big)\bigg) \rho_{t}(x) \de x. 
\end{align}
Using log-Sobolev inequality, we get 
\begin{align*}
    -\int_{\R^{d}}\bigg( \nabla \ln\Big( \frac{\rho_{t}(x)}{\pi(x)}\Big) \cdot \C(\mu_{t})\nabla \ln\Big( \frac{\rho_{t}(x)}{\pi(x)}\Big)\bigg) \rho_{t}(x) \de x
   & \leq 
  - \lambda_{\min}^{\C}(t) \int_{\R^{d}} \Big| \nabla \ln\Big( \frac{\rho_{t}(x)}{\pi(x)}\Big)\Big|^{2} \rho_{t}(x) \de x
  \\ & 
  \leq  
  - \frac{\lambda_{\min}^{\C}(t)}{2\lambda_{\ls}} \h(\rho_{t}\; | \;\pi).
\end{align*}
Therefore,
\begin{align*}
     \frac{\de }{\de t}\h(\rho_{t}\; | \;\pi) \leq - \frac{\lambda_{\min}^{\C}(t)}{2\lambda_{\ls}\beta} \h(\rho_{t}\; | \;\pi)
\end{align*}
which on applying Lemma~\ref{min_eigen_bound_lemma}, we ascertain
\begin{align}
 \frac{\de }{\de t}\h(\rho_{t}\; | \;\pi) \leq
 -\min\bigg(\frac{1}{\beta\big(   2L_{A} + \beta L_{V}^{2}\big)}, \lambda^{\C}_{\min}(0)\bigg) \frac{1}{2\lambda_{\ls}\beta}\h(\rho_{t}\; | \;\pi).
\end{align}
\end{proof}

\subsection{Proof of Lemma~\ref{max_eigen_bound_lemma}} 


\begin{proof}[Proof of Lemma~\ref{max_eigen_bound_lemma}]
From (\ref{n_langevin_eqn_2.3}), we have
 \begin{align*}
        \de \C_t &= 
    - \E[ (X(t) - \E X(t)) (\nabla U(X(t)) - \E \nabla U(X(t)))^{\top}] \C_{t} \de t   
    \\  & \quad 
    - \C_{t}\E[ (\nabla U(X(t)) - \E \nabla U(X(t)))(X(t) - \E X(t))^{\top}]\de t + \frac{2}{\beta}\C_{t}\de t. \numberthis 
    \end{align*}
which due to Assumption~\ref{first_assum_on_U} gives
\begin{align*}
     \frac{\de}{\de t} \C_t &=  -2 \C_t A \C_t  -  \E[Z_t P_t^{\top}] \C_{t}   
    - \C_{t}\E[ P_t Z_t^{\top}] + \frac{2}{\beta}\C_{t}, 
 \end{align*}
where $Z_t = X(t) - \E X(t)$ and $ P_t = \nabla V(X(t)) - \E\nabla V(X(t))$. 

For $|y|^{2} = 1$, we have
\begin{align*}
    \frac{\de }{\de t} y^{\top} \C_{t} y & =  -2y^{\top} \C_t A \C_t y  -  y^{\top}\E[Z_t P_t^{\top}] \C_{t} y   
    - y^{\top}\C_{t}\E[ P_t Z_t^{\top}]y + \frac{2}{\beta}y^{\top}\C_{t}y 
    \\  &  
    \leq -2\ell_{A}y^{\top} \C_{t}^{2} y  -  2y^{\top}\E[Z_t P_t^{\top}] \C_{t} y   
    + \frac{2}{\beta}y^{\top}\C_{t}y.   \numberthis \label{n_langevin_2.23}
\end{align*}
First consider the second term on the right hand side of the inequality:
\begin{align*}
    - 2y^{\top} \E[Z_t P_t^{\top}]\C_{t}y & = -2\E[y^{\top} Z_t P_t^{\top}\C_{t} y ] \leq 2 (\E |y^{\top} Z_t|^{2})^{1/2} (\E| P_t^{\top} \C_{t}y|^{2})^{1/2} 
    \\  & 
    =  2(y^{\top}\C_{t} y)^{1/2} ( y^{\top} \C_{t} \E [P_t P_t^{\top}] \C_{t}y )^{1/2}.  
\end{align*}
Due to our assumption on $V$, there exists a $K_{V}>0$ such that following holds for all $t >0$ and $z\in \R^d$:
\begin{align}
    z^{\top} P_t P_t^{\top} z \leq  K_{V} |z|^{2}.
\end{align}
Explicit $K_V$ calculation $
    z^{\top} P_t \leq  2 L_V |z| $
therefore $K_V = 4L_V^2$.

An application of Young's inequality gives
\begin{align*}
    - 2y^{\top} \E[Z_t P_t^{\top}]\C_{t}y &
    \leq  2\sqrt{K_{V}}(y^{\top}\C_{t} y)^{1/2} ( y^{\top} \C_{t}^{2}y )^{1/2}
    \leq  \frac{\ell_{A}}{3}y^{\top} \C_{t}^{2}y +  \frac{3 K_{V}}{ \ell_{A}}y^{\top} \C_{t}y.  
\end{align*}
Hence, using above inequality in (\ref{n_langevin_2.23}), we get
\begin{align*}
     \frac{\de }{\de t} y^{\top} \C_{t} y & 
    \leq -\frac{5}{3}\ell_{A}y^{\top} \C_{t}^{2} y + \frac{3 K_{V}^{2}}{ \ell_{A}}y^{\top} \C_{t}y  
    + \frac{2}{\beta}y^{\top}\C_{t}y. \numberthis \label{n_langevin_2.25}
\end{align*}

Let us now  analyze the term $y^{\top} \C_{t}^{2} y $ appearing in (\ref{n_langevin_2.23}). We have already proven uniform in time lower bound on eigenvalues of $\C_{t}$. We write the eigenvalue decomposition of matrix $\C_t$ as $O_{t}D_{t}O_{t}^{\top}$, where $O_{t}$ is orthogonal matrix and $D_{t}$ is a diagonal matrix whose diagonals are given by eigenvalues of $\C_t$. 
This implies we can write
\begin{align}
    y^{\top}\C_{t}^{2} y =  y^{\top}O_t D_{t}^{2} O_{t}^{\top} y. 
\end{align}
We denote $v_{t} = O_{t}^{\top} y $ and it is clear that $|v_{t}|^{2} = y^{\top} O_{t}^{\top} O_{t} y = 1$. The calculations below follow due to above properties:
\begin{align*}
    (y^{\top} \C_{t} y)^{2} & =  (v_{t}^{\top} D_{t} v_{t})^{2} =  \bigg(\sum_{i}^{d} \lambda^{\C}_{i}(t) (v_{t}^{i})^{2}\bigg)^{2} 
    \leq \sum_{i,j=1}^{d} \lambda^{\C}_{i}(t) (v_{t}^{i})^{2} \lambda^{\C}_{j}(t) (v_{t}^{j})^{2},
\end{align*}
where $\lambda^{\C}_{i}(t)$ denotes the $ii$-th component of $D_{t}$ and $v_{t}^{i}$ represents the $i$-th component of $v_{t}$. On applying Young's inequality, we arrive at the following:
\begin{align*}
   (y^{\top} \C_{t} y)^{2} &\leq   \frac{1}{2}\sum_{i,j=1}^{d} \big[ \big(\lambda^{\C}_{i}(t) \big)^{2}  + \big(\lambda^{\C}_{j}(t) \big)^{2}\big]  (v_{t}^{i})^{2} (v_{t}^{j})^{2}
   \\ &
   = \frac{1}{2}\sum_{i,j=1}^{d} \big(\lambda^{\C}_{i}(t) \big)^{2}(v_{t}^{i})^{2} (v_{t}^{j})^{2} + \frac{1}{2}\sum_{i,j=1}^{d} \big(\lambda^{\C}_{j}(t) \big)^{2}(v_{t}^{j})^{2} (v_{t}^{i})^{2}
   \\  & 
   =\sum_{i}^{d} \big(\lambda^{\C}_{i}(t) \big)^{2}(v_{t}^{i})^{2}
   = v_{t}^{\top} D_{t}^{2} v_{t}  
   = y^{\top} O_{t} D_{t}^{2} O_{t}^{\top}y 
   = y^{\top} \C_{t}^{2} y. \numberthis 
\end{align*}
Consequently, we have obtained that
\begin{align}
    (y^{\top} \C_{t} y)^{2} \leq y^{\top} \C_{t}^{2} y,
\end{align}
for all $y$ such that $|y| =1$. Using the estimate, $
    -(y^{\top} \C_{t} y)^{2} \geq  - y^{\top} \C_{t}^{2} y 
$
in (\ref{n_langevin_2.25}), we get
\begin{align*}
     \frac{\de }{\de t} y^{\top} \C_{t} y & 
    \leq -\frac{5}{3}\ell_{A}(y^{\top} \C_{t} y)^{2} + \frac{3 K_{V}}{ \ell_{A}}y^{\top} \C_{t}y  
    + \frac{2}{\beta}y^{\top}\C_{t}y. \numberthis 
\end{align*}
Therefore, we have
\begin{align*}
    \frac{\de }{\de t} y^{\top} \C_{t} y & 
    \leq - \ell_{A}(y^{\top} \C_{t} y)^{2} + 3^{3}\bigg( \frac{ K_{V}}{ \ell_{A}}\bigg)^{2} + 
    \frac{3}{\beta^{2}},
\end{align*}
where we have utilized Young's inequality as
\begin{align*}
    \frac{3 K_{V}}{4 \ell_{A}}y^{\top} \C_{t}y  & \leq 
    \frac{1}{3}(y^{\top} \C_{t} y)^{2} +  3^{3}\bigg( \frac{ K_{V}}{ \ell_{A}}\bigg)^{2},
    \\  
    \frac{2}{\beta} y^{\top}\C_{t} y & \leq  \frac{1}{3} (y^{\top}\C_{t}y)^{2}  +  \frac{3}{\beta^{2}}.
 \end{align*}
Therefore, also using $ -x^2 \leq -x +1  $, we get
\begin{align*}
     \frac{\de }{\de t} y^{\top} \C_{t} y & 
    \leq - \ell_{A}(y^{\top} \C_{t} y)  + \ell_A + 3^{3}\bigg( \frac{ K_{V}}{ \ell_{A}}\bigg)^{2} + 
    \frac{3}{\beta^{2}}
\end{align*}
Denoting $
    \tilde{\Lambda} (t,y) = y^{\top}\C_{t} y
$
and using $e^{\ell_{A}t}$ as integrating factor, we get 
\begin{align*}
    \frac{\de}{\de t} e^{\ell_{A} t }\tilde{\Lambda }(t,y) \leq  \Bigg(\ell_A + \bigg(\frac{3}{4}\bigg)^{3}\bigg( \frac{ K_{V}}{ \ell_{A}}\bigg)^{2} + 
    \frac{3}{\beta^{2}}\Bigg) e^{\ell_{A}t}.
\end{align*}
Therefore, we have
\begin{align*}
    \tilde{\Lambda}(t,y) \leq \tilde{\Lambda}(0,y)e^{-\ell_{A} t } + \frac{1}{\ell_{A}}\Bigg(\ell_A + \bigg(\frac{3}{4}\bigg)^{3}\bigg( \frac{ K_{V}}{ \ell_{A}}\bigg)^{2} + 
    \frac{3}{\beta^{2}}\Bigg) (1 - e^{-\ell_{A}t}).
\end{align*}
Taking supremum over $y$, we finally get our estimate 
\begin{align}
\lambda^{\C}_{\max}(t) \leq \lambda^{\C}_{\max}(0)e^{-\ell_{A} t } + \frac{1}{\ell_{A}}\Bigg(\bigg(\frac{3}{4}\bigg)^{3}\bigg( \frac{ K_{V}}{ \ell_{A}}\bigg)^{2} + 
    \frac{3}{\beta^{2}}\Bigg) (1 - e^{-\ell_{A}t}).
\end{align}
\end{proof}

\subsection{Proof of Theorem~\ref{lem:X_bounded_moments}}\label{unif_mome_bound}

\begin{proof}[Proof of Theorem~\ref{lem:X_bounded_moments}]
Below we prove the result for $l\geq 2$, and after which the result for $ 0<l<2$ follows from the Cauchy-Bunyakovsky-Schwartz inequality. 

Using Ito's formula for $U^{l}(x)$, with $l \geq 2$, we have
\begin{align*}
    \de U^{l}(X(t)) & = - l U^{l-1}(X(t)) \la\nabla  U(X(t)) \cdot \C_{t} \nabla U(X(t))\ra \de t + \frac{l}{\beta}U^{l-1}(X(t)) \trace\big[\nabla^{2}U(X(t)) \C_{t}\big] \de t 
    \\  & \quad 
    + \frac{l}{\beta}(l-1) U^{l-2}(X(t))\la\nabla U(X(t)) \cdot \C_t \nabla U(X(t))\ra \de t + \sqrt{\frac{2}{\beta}} l U^{l-1} \la \nabla U(X(t)) \cdot \sqrt{\C_t} \de W(t)\ra. 
\end{align*}

Due to boundedness of second derivatives of $U$ and uniform in time bounds on $\lambda_{\min}^{\C}(t)$ and $\lambda_{\max}^{\C}(t)$ (see Lemma~\ref{min_eigen_bound_lemma} and Lemma~\ref{max_eigen_bound_lemma}), we have
\begin{align}
    \trace\big[\nabla^{2}U(X(t)) \C_{t}\big]  &\leq  B, \\ 
    -\la \nabla U(X(t))\cdot \C_t \nabla U(X(t))\ra & \leq -\ell_{\C} |\nabla U(X(t))|^2,  \\
    \la \nabla U(X(t))\cdot \C_t \nabla U(X(t))\ra & \leq L_{\C} |\nabla U(X(t))|^2,
\end{align}
where $B >0$ is a constant independent of $t$, and $\ell_\C$ and $L_\C$ are from (\ref{notation_l_sigma}) and (\ref{notation_L_sigma}), respectively. This leads to the following:
\begin{align*}
     \de U^{l}(X(t)) & = - l \ell_\C U^{l-1}(X(t)) |\nabla U(X(t))|^2  \de t + \frac{l}{\beta}B U^{l-1}(X(t))  \de t 
    \\  & \quad 
    + \frac{l}{\beta}(l-1)L_\C U^{l-2}(X(t))|\nabla U(X(t))|^2  \de t + \sqrt{\frac{2}{\beta}} l U^{l-1} \la \nabla U(X(t)) \cdot \sqrt{\C_t} \de W(t)\ra
\end{align*}
which on taking expectation on both sides gives
\begin{align}
     \de \E U^{l}(X(t)) & = - l \ell_\C \E U^{l-1}(X(t)) |\nabla U(X(t))|^2  \de t + \frac{l}{\beta}B \E U^{l-1}(X(t))  \de t 
  \nonumber  \\  & \quad 
    + \frac{l}{\beta}(l-1)L_\C \E U^{l-2}(X(t))|\nabla U(X(t))|^2  \de t.
\end{align}

Next, we derive upper and lower bounds for $|\nabla U(X(t))|^2$ in terms of $U(x)$. To this end, since $V$ by Assumption~\ref{first_assum_on_U} is Lipschitz continuous, we have 
$ V(x) \leq V(0) + L_V |x|$. Therefore,
\begin{align*}
    U(x) & \leq \frac{1}{2}L_{A} |x|^2 + L_V |x| + V(0),
\end{align*}
which, upon using Young's inequality $L_V|x| \leq  \frac{1}{2}|x|^2 + \frac{1}{2}L_V^2$, gives
\begin{align*}
    U(x) &\leq  \frac{1}{2}(L_{A} +1)|x|^2 + \frac{1}{2}L_V^2 + V(0), 
\end{align*}
This implies 
\begin{align}
    |x|^2 \geq  \frac{2}{L_A + 1}U(x) - \frac{1}{L_A + 1}\big(L_V^2 + 2V(0)\big). \label{estimate_on_x}
 \end{align}
 Due to our assumption on $U$, i.e., Assumption~\ref{first_assum_on_U}, and (\ref{estimate_on_x}), we obtain
\begin{align}
    |\nabla U(x)|^{2} &=  |Ax + \nabla V(x)|^2 \geq  \frac{1}{2}|Ax|^2 - |\nabla V(x)|^2 \geq \frac{\ell^2_A}{2} |x|^2 - L_V^2 \nonumber
    \\   &  \geq  \frac{ \ell_A^2}{L_A + 1} U(x) + \ell_A^2\frac{2V(0) + L_V^2}{2(L_A + 1)}- L_V^2 .  \label{eqn_new_eld_2.33}
\end{align} 
To derive the upper bound we again using Lipschitz continuity of $V$, which gives that
\begin{align*}
    U(x) \geq \frac{1}{2}\ell_A^2|x|^2 + V(0) - L_V|x| \geq \frac{1}{4}\ell_A^2|x|^2 - \frac{L_V^2}{\ell_A^2} + V(0),
\end{align*}
where we have used Young's inequality, i.e., $ L_V |x| \leq \ell_A^2|x|^2/4 + L^2_V/\ell_A^2 $, in last step. Hence,
\begin{align*}
    |x|^2 \leq \frac{4}{\ell_A^2}U(x)  + \frac{4L_V^2}{\ell_A^4} - \frac{4}{\ell_A^2}V(0).
\end{align*}
This implies
\begin{align}
    |\nabla U(x)|^2 = |Ax + \nabla V(x)|^2 \leq 2 L_A^2 |x|^2 + 2 L_V^2 \leq  \frac{8 L_A^2}{\ell_A^2}U(x)  + \frac{8 L_A^2 L_V^2}{\ell_A^4} - \frac{8L_A^2}{\ell_A^2}V(0) + 2L_V^2. \label{eqn_new_eld_2.35}
\end{align}

Using (\ref{eqn_new_eld_2.33}) and (\ref{eqn_new_eld_2.35}), we have
\begin{align*}
    \de \E U^{l}(X(t)) & \leq - l  \ell_{\C}\frac{ \ell_A^2}{L_A + 1} \E U^{l}(X(t)) \de t - l\ell_{\C}\ell^2_A \frac{2 V(0) + L_V^2}{2(L_A + 1)}\E U^{l-1}(X(t))\de t \\  & \quad + l\ell_{\C}L^2_V\E U^{l-1}(X(t)) \de t + \frac{l}{\beta}B \E U^{l-1}(X(t))  \de t 
        + \frac{l}{\beta}(l-1)L_{\C} \frac{8L_A^2}{\ell_A^2}\E U^{l-1}(X(t)) \de t \\  & \quad  + \frac{l}{\beta}(l-1)L_{\C}\Big(\frac{8 L_A^2 L_V^2}{\ell_A^4} - \frac{8L_A^2}{\ell_A^2}V(0) + 2L_V^2  \Big) \E U^{l -2}(X(t)) \de t 
        \\ &
        \leq -\eta_1 \E U^{l }(X(t)) \de t + \eta_2 \E U^{l -1}(X(t)) \de t + \eta_3 \E U^{l-2}(X(t)) \de t,
\end{align*}
where 
\begin{align}
    \eta_1 &:= l  \ell_{\C}\frac{ \ell_A^2}{L_A + 1}, \\
    \eta_2 &:=  l\ell_{\C}L^2_V + \frac{l}{\beta}B + \frac{l}{\beta}(l-1)L_{\C} \frac{8L_A^2}{\ell_A^2}, \\
    \eta_3 &:= \frac{l}{\beta}(l-1)L_{\C}\Big(\frac{8 L_A^2 L_V^2}{\ell_A^4} - \frac{8L_A^2}{\ell_A^2}V(0) + 2L_V^2  \Big).
\end{align}
Using Young's inequality, we have
\begin{align*}
    \eta_2 U^{l-1}(x) &\leq \frac{\eta_1}{3} U^{l}(x) + \Big(\frac{2 \eta_1 ( l -1) }{3 l}\Big)^{-(l-1)/l}\frac{\eta_2^l}{2 l}  ,\\ 
    \eta_3 U^{l-2}(x) & \leq \frac{\eta_1}{3} U^{l}(x)   + \Big(\frac{2 \eta_1 ( l -2) }{3 l}\Big)^{-(l-2)/l}\frac{\eta_3^{l/2}}{ l},
\end{align*}
which results in 
\begin{align*}
    \de \E U^{l}(X(t)) 
    \leq - \frac{\eta_1}{3} \E U^{l}(X(t))\de t  +  \eta_4 \de t,
\end{align*}
where $\eta_4$ is independent of $t$ and given by
\begin{align*}
   \eta_4 =  \Big(\frac{2 \eta_1 ( l -1) }{3 l}\Big)^{-(l-1)/l}\frac{\eta_2^l}{2 l}+ \Big(\frac{2 \eta_1 ( l -2) }{3 l}\Big)^{-(l-2)/l}\frac{\eta_3^{l/2}}{ l}.
\end{align*}
This implies
\begin{align*}
    \E U(X(t)) \leq  \E U(X(0))e^{-\frac{\eta_1}{3}t} + 3 \frac{\eta_4}{\eta_1} (1 - e^{-\frac{\eta_1}{3}t}). 
\end{align*}

This implies uniform in time moment bound, i.e.,
\begin{align*}
    \E |X(t)|^{2 l} \leq C_{M},
\end{align*}
where $C_{M}>0$ is a constant independent of $t$. 
\end{proof}

\section{Particle approximation and propagation of chaos}\label{sec:particle_approx}

For a process $X(t)$, drive by the dynamics \eqref{eq:ek_sdes}, the time marginal law, denoted by $\mu_{t} := \mathcal{L}^{X}_{t}$,
satisfies the non-linear Fokker-Planck PDE
\begin{align}\label{pde_mf}
    \frac{\partial \mu_t}{\partial t}(t,x) =  \la \nabla_{x} \mu_t \cdot \Sigma(\mu_t)\nabla U(x)\ra &+ \frac{1}{\beta  } \trace(\C(\mu_{t})\nabla^2 \mu_t),\;\; (t,x) \in [0,T]\times \mathbb{R}^{d}.
 \end{align}
In this section, we prove that we can use a particle approximation to approximate this time marginal law. By the contraction result in Theorem~\ref{main_thrm_exp_conv}, this means that we can use a particle approximation to approximate $\pi(x)\de x$, with $\pi$ given in \eqref{eq:pi}.

To this end, let $N \in \mathbb{N}$ denote the number of particles and let $X^{i,N}(t)$ represent the position of the $i$-th particle at time $t$.  
Let $\mathcal{E}^{N}(t)$ represent the empirical measure defined as
\begin{align}
    \mathcal{E}^{N}(t) = \frac{1}{N}\sum_{j=1}^{N}\delta_{X^{j,N}(t)}.
\end{align}
 Consider the $d \times N$ deviation matrix 
 \[
 Q^{N}(t) = [X^{1,N}(t)-\M^N(t),...,X^{N,N}(t)-\M^N(t)]
 \]
 with $\M^N_t$ being the ensemble mean given by $
 \M^N_t:=   \M(\cE^N(t)) = \frac{1}{N}\sum_{j=1}^N X^{j,N}(t)$. 
The ensemble covariance, denoted as $\C^{N}_t$, is defined as
\begin{equation}
   \C^{N}_t :=  \C(\mathcal{E}^{N}(t)) = \frac{1}{N} Q^{N}(t) Q^{N}(t)^{\top}.
    \label{non_langevin_ips_ensemble_Covariance}
\end{equation}
One possible particle approximation to the mean-field limit is given by the following SDEs governing the interacting particle system:
\begin{align}
    \de X^{i,N}(t) = - \C(\mathcal{E}^{N}(t))\nabla U(X^{i,N}(t))\de t + \sqrt{\frac{2}{\beta}}\sqrt{\C(\cE^{N}(t))}\de W^i(t),
    \label{non_langevin_sde_strong_ips}
\end{align}
where $(W^{i}(t))_{t \geq 0}$, $i =1, \dots, N,$ denote standard $d$-dimensional Wiener processes.
The strong convergence of (\ref{non_langevin_sde_strong_ips}) to the mean-field limit (\ref{eq:ek_sdes}) is shown in \cite{ding2021ensemble_sampler} and \cite{vaes2024sharpchaos} in linear and non-linear setting, respectively. However, an implementation of (\ref{non_langevin_sde_strong_ips}) requires the computation of the square root of a $d\times d$ matrix which may be computationally expensive.
To overcome this,  we first note that 
for the purpose of sampling and computation of ergodic averages, a weak sense approximation will suffice. To this end, inspired from \cite{garbuno2020affine}, we consider a system of SDEs which approximates mean-field SDEs in large particle limit, but in weak-sense. 

Let $(B^{i}(t))_{t \geq 0}$, $i = 1,\dots,N$, be standard $N$-dimensional Wiener processes. We consider the following SDEs driving the interacting particle system:
\begin{align}
    \de X^{i,N}(t) = - \C(\mathcal{E}^{N}(t))\nabla U(X^{i,N}(t))\de t + \sqrt{\frac{2}{\beta N}}Q^{N}(t)\de B^i(t),
    \label{non_langevin_sde_weak_ips}
\end{align}
where $X^{i,N}(t)  \in \mathbb{R}^{d}$ denotes the position of $i-$th particle at time $t$.

\begin{remark}[Derivative-free approach]\label{remark_3.1_deri_free} Using the following approximate first-order linearization \cite{iglesias2013ensemble, schillings2017analysis_enkf}:
\begin{align*}
U(X^{j,N}(t)) &\approx U(X^{i,N}(t)) + \la (X^{j,N}(t) - X^{i,N}(t)) \cdot \nabla U(X^{i,N}(t))\ra,   \\
U(X^{k,N}(t)) &\approx U(X^{i,N}(t)) + \la (X^{k,N}(t) - X^{i,N}(t)) \cdot \nabla U(X^{i,N}(t))\ra,  
\end{align*}
we get $
   U(X^{j,N}(t))  - \bar{U}^{N}(t) \approx  \la (X^{j,N}(t) - \M^{N}(t)) \cdot \nabla U(X^{i,N}(t))\ra
  $, where $\bar{U}^{N}(t) := \frac{1}{N}\sum_{k=1}^{N}U(X^{k,N}(t))$. The above, with some more computations, lead to the following derivative-free dynamics:
  \begin{align}
  \de Z^{i,N}(t) = -\frac{1}{N}\sum_{j=1}^{N}(Z^{j,N}(t) - \M^N(t)) \big(U(Z^{j,N}(t))  - \bar{U}^{N}(t)\big)\de t + \sqrt{\frac{2}{\beta N}}Q^{N}(t)\de B^i(t).
  \end{align}
\end{remark}

In order to prove propagation of chaos for \eqref{non_langevin_sde_weak_ips}, we impose the following assumption on $U$.

\begin{assumption}\label{second_assum_on_U}
Let $U \in C^2(\mathbb{R}^d)$ and $U(x) \geq 0$ for all $x \in \R^d$. We assume that there exists a compact set $\mathcal{K}$ such that following holds for some $K_1, K_2, K_3, K_4 >0$:
\begin{align*}
     K_1  |x|^{2} \leq U(x) \leq K_2 |x|^{2} \quad \text{and}\quad K_3|x| \leq |\nabla U(x)| \leq K_4|x|, \quad \text{for all } \quad x \in \mathcal{K}^c.
 \end{align*}
 We also assume that all second derivatives of $U$ are uniformly bounded in $x$. 
\end{assumption}

In \cite{vaes2024sharpchaos}, the well-posedness of (\ref{el_sde_mf}) as well as (\ref{non_langevin_sde_strong_ips}) is proved under Assumption~\ref{second_assum_on_U}. In the proof of well-posedness of particle system (\ref{non_langevin_sde_strong_ips}) in \cite{vaes2024sharpchaos}, the Khasminski-Lyapunov approach is employed using the same Lyapunov function as used in \cite{garbuno2020affine}. The exact same arguments can be used to show well-posedness of (\ref{non_langevin_sde_weak_ips}). We avoid repeating the same calculations considering the line of arguments remains unchanged since $ \C^N = \frac{1}{N}Q^{N}(Q^{N})^{\top} = \sqrt{\C^{N}}\sqrt{\C^{N}}^{\top}$. In addition, one can also obtain the moment bounds uniform in $N$ as has been done in \cite{vaes2024sharpchaos}, i.e., for $p \geq 1$ it holds that
\begin{align}
    \sup_{i =1,\dots, N} \sup_{t\in [0,T]}\mathbb{E}|X^{i,N}(t)|^{2p} &\leq C, \label{eld_opt_eq_mb}
 \end{align}
where $C >0$ is independent of $N$. As a consequence of uniform in $N$ moment bound, we have
\begin{align}
    \mathbb{E}|\Sigma(\mathcal{E}_{N}(t))| &= \frac{1}{N}\E \bigg|\sum_{i=1}^{N }(X^{i,N}(t) -\M^{N}(t)) (X^{i,N}(t) -\M^{N}(t))^{\top}\bigg|  \nonumber  \\  &
    \leq   \frac{1}{N}\sum_{i=1}^{N} \E|X^{i,N}(t) -\M^{N}(t)|^2
   \nonumber  \\ &  \leq  \frac{2}{N}\sum_{i=1}^{N} \E|X^{i,N}(t)|^2 + 2\E|\M^{N}(t)|^2 \leq  4 \frac{1}{N}\sum_{i=1}^{N}\mathbb{E}|X^{i,N}(t)|^{2} \leq C, \label{elo_cov_bound}
\end{align}

The main result of this section is the following theorem.

\begin{theorem}\label{weak_prop_chaos}
    Let $\mu_{0} \in \mathcal{P}_{2}(\mathbb{R}^{d})$, and let Assumption~\ref{second_assum_on_U} hold. Moreover, let $(X^{i,N})_{i=1}^{N}$ be the solution of interacting particle system (\ref{non_langevin_sde_weak_ips}), and let $\mathcal{E}^{N}_t$ denote empirical measure of these interacting $N$ particles at time $t$. Then, for all $t \in [0,T]$, there exists a deterministic limit $\mu_t$  of the empirical measure $\mathcal{E}^{N}_t$ as $ N \rightarrow \infty$. This $\mu_t$ satisfies the mean-field PDE (\ref{pde_mf}) in weak sense. 
\end{theorem}

\subsection{Proof of Theorem~\ref{weak_prop_chaos}}
The proof follows the classical trilogy of arguments (see \cite{sznitman1991topics}):
\begin{itemize}
   \item Tightness of law of random empirical measure: Together with Prokhorov's theorem, this ensures that there exists a converging subsequence of the law of the random empirical measure as the number of particles tend to infinity.
   \item Identification of limit: This is achieved by looking at the PDE corresponding to mean-field SDEs.
   \item Uniqueness of limit: This follows from the uniqueness of solution of mean-field SDEs. 
\end{itemize}

\subsubsection{Tightness of random measure}
Let $\mathcal{L}(\mathcal{E}^{N})$ be law of empirical measures which means it belongs to $\mathcal{P}(\mathcal{P}(C([0,T];\mathbb{R}^{d})))$. We aim to show that the family of measures $(\mathcal{L}(\mathcal{E}^{N}))_{N\in \mathbb{N}}$ is tight in $\mathcal{P}(\mathcal{P}(C([0,T];\mathbb{R}^{d})))$.
From \cite[Proposition~2.2]{sznitman1991topics}, showing tightness of family of measures $(\mathcal{L}(\mathcal{E}^{N}))_{N\geq 2}$ in $\mathcal{P}(\mathcal{P}(C([0,T];\mathbb{R}^{d})))$ is equivalent to showing that family  $ (\mathcal{L}(X^{1,N}))_{N\in \mathbb{N}}$ is tight in $\mathcal{P}(C([0,T];\mathbb{R}^{d}))$. To this end, we employ Aldous criteria (see \cite[Section~16]{billingsley2013convergence}) as follows:
\begin{itemize}
    \item[(i)] For all $t\in [0,T]$, the time marginal law of $(X^{i,N}(t))_{i \in \mathbb{N}}$ is tight as a sequence in space $\mathcal{P}(\mathbb{R}^{d})$.
    \item[(ii)] For all $\epsilon > 0$ and $\zeta >0$, there exist $n_{0}$ and $\delta > 0 $ such that for any sequence of stopping times $\tau_{j}$ 
    \begin{align}
        \sup_{j > n_{0}} \sup_{\theta \leq \delta}\mathbb{P}(|X^{i,N}_{\tau_{j} + \theta} - X^{i,N}_{\tau_{j} }| > \zeta ) \leq \epsilon.
    \end{align}
\end{itemize}
From (\ref{eld_opt_eq_mb}), there is a $K$, not depending on $N$, such that $\E|X^{1,N}(t)|^2 \leq K$.
To verify the first condition consider a compact set $\mathcal{K}_{\epsilon} : =  \{ y \; ; \; |y|^{2} \leq K/\epsilon\}  $. Using Markov's inequality, we have
\begin{align}
\mathcal{L}_{X^{1,N}(t)}\big( \mathcal{K}_{\epsilon}^{c} \big) = \mathbb{P} ( |X^{1,N}(t)|^{2} > K/\epsilon ) \leq \epsilon\frac{\mathbb{E}|X^{1,N}(t)|^{2}}{K}  \leq \epsilon.
\end{align}
This verifies the first condition that for each $t \in [0,T]$ the sequence $(\mathcal{L}_{X^{1,N}(t)})_{N\in \mathbb{N}}$ is tight. 
To verify the second condition, we start with
\begin{align}
    X^{1,N}(\tau_{j} +\delta) = X^{1,N}(\tau_{j} ) -
    \int_{\tau_{j}}^{\tau_{j} +\delta} \Sigma(\mathcal{E}^{N}(t))
    \nabla U(X^{1,N}(t)) \de t +  \int_{\tau_{j}}^{\tau_{j} +\delta}\sqrt{\frac{2}{\beta  N}}Q^{N}(t) \de 
    B^{1}(t).
\end{align}
Rearranging terms, and squaring  and taking expectation on both sides, we obtain
\begin{align*}
    \mathbb{E}|X^{1,N}(\tau_{j} +\delta) - X^{1,N}(\tau_{j} )|^{2} \leq  
    2\mathbb{E}\bigg|\int_{\tau_{j}}^{\tau_{j} +\delta}  \Sigma(\mathcal{E}^{N}(t))
    \nabla U(X^{1,N}(t)) \de t  \bigg|^{2}
  + 2\mathbb{E}\bigg|\int_{\tau_{j}}^{\tau_{j} +\delta}\sqrt{\frac{2}{\beta  N}} Q^{N}(t) \de B^{1}(t)\bigg|^{2}.
\end{align*}

Considering the first term, by Cauchy-Bunyakovsky-Schwarz inequality, we have
\begin{align}
\mathbb{E}\bigg|\int_{\tau_{j}}^{\tau_{j} +\delta}  \Sigma(\mathcal{E}^{N}(t))
    \nabla U(X^{1,N}(t)) \de t  \bigg|^{2} & \leq \delta\mathbb{E} \bigg(\int_{\tau_{j}}^{\tau_{j} +\delta}  |\Sigma(\mathcal{E}^{N}(t))
    \nabla U(X^{1,N}(t))|^2 \de t  \bigg) \nonumber \\
    & \leq \delta \sup_{t \in[0,T]}\mathbb{E}  |\Sigma(\mathcal{E}^{N}(t)) \nabla U(X^{1,N}(t))|^{2}. 
\end{align}

Considering the second term, from Ito's isometry, we have
\[
    \mathbb{E}\bigg|\int_{\tau_{j}}^{\tau_{j} +\delta}\sqrt{\frac{2}{\beta  N}} Q^{N}(t) \de B^{1}(t)\bigg|^{2} = \mathbb{E} \int_{\tau_{j}}^{\tau_{j} +\delta} \bigg|\frac{2}{\beta  N} Q^{N}(t)(Q^{N})^{\top}(t) \bigg|^{2}  \de t = \mathbb{E}\int_{\tau_{j}}^{\tau_{j} +\delta}\bigg|  \frac{2}{\beta  } \Sigma(\mathcal{E}^{N}(t))\bigg| ^{2} \de t.
\] 
Now, by Cauchy-Bunyakovsky-Schwarz inequality, we have
\begin{align}
    \mathbb{E}\int_{\tau_{j}}^{\tau_{j} +\delta}\bigg|  \frac{2}{\beta  } \Sigma(\mathcal{E}^{N}(t))\bigg| ^{2} \de t & \leq \E\bigg( \int_{\tau_{j}}^{\tau_{j} +\delta} \frac{4}{\beta^2} \de t \bigg)^{1/2}  \bigg(\int_{\tau_{j}}^{\tau_{j} +\delta} \E|\C(\cE^{N}(t))|^4\de t\bigg)^{1/2} \nonumber \\
    & \leq  C\delta^{1/2}  \bigg(\int_{0}^{T}\E|\C(\cE^{N}(t))|^4\de t\bigg)^{1/2} \leq C \delta^{1/2}  \sup_{t \in [0,T]}\E|\C(\cE^{N}(t))|^2.
\end{align}
This implies
\begin{align*}
    \mathbb{E}|X^{1,N}(\tau_{j} +\delta) - X^{1,N}(\tau_{j} )|^{2} &\leq  
  C \delta\mathbb{E}\int_{\tau_j}^{\tau_j +\delta}  |\Sigma(\mathcal{E}^{N}(t))
    \nabla U(X^{1,N}(t))|^{2} \de t  
     + C\delta^{1/2} \sup_{t \in [0,T]} \E|\C(\cE^{N}(t))|^2
   \\   &   \leq  
  C \delta \sup_{t \in[0,T]}\mathbb{E}  |\Sigma(\mathcal{E}^{N}(t))
    \nabla U(X^{1,N}(t))|^{2}  
        + C\delta^{1/2} \sup_{t \in [0,T]} \E|\C(\cE^{N}(t))|^2.  
\end{align*}
Using moment bounds from (\ref{eld_opt_eq_mb}), we get
\begin{align}
    \mathbb{E}|X^{1,N}(\tau_{j} +\delta) - X^{1,N}(\tau_{j} )|^{2} &\leq C\delta^{1/2}. \label{eld_eqn_4.10}
\end{align}

It is not difficult to see that  for all $\epsilon >0$ and $\zeta$ there exists a $\delta_{0}$ such that for all $\delta \in(0, \delta_{0})$, we have
\begin{align}
    \sup_{\delta \in (0, \delta_{0})} \mathbb{P}|X^{1,N}(\tau_{j} + \delta) -X^{1,N}(\tau_{j}) | > \zeta ) \leq \sup_{\delta \in (0, \delta_{0})}\frac{\mathbb{E}|X^{1,N}(\tau_{j} +\delta) - X^{1,N}(\tau_{j} )|^{2}}{\zeta^{2}} \leq \epsilon,
\end{align}
where we have applied Markov's inequality. It is obvious that the choice of $\delta_{0}$ depends on positive constant $C$ appearing on the right hand side of (\ref{eld_eqn_4.10}). This ensures that second condition in Aldous criteria holds.

Having established the tightness of $(\mathcal{L}_{X^{1,N}})_{N\in \mathbb{N }} $ in $\mathcal{P}(C([0,T]; \mathbb{R}^{d}))$, and hence the tightness of $(\mathcal{L}(\mathcal{E}^{N})_{N\in \mathbb{N}}$ in $\mathcal{P}(\mathcal{P}(C([0,T]; \mathbb{R}^{d})))$, using Prokhorov's theorem \cite{billingsley2013convergence} we can infer that there exist a sub-sequence $(\mathcal{L}(\mathcal{E}^{N_{k}})_{N_{k}\in \mathbb{N}}$ of $(\mathcal{L}(\mathcal{E}^{N}))_{N\in \mathbb{N}}$ and a random measure $\mu $ such that $\mathcal{L}(\mathcal{E}^{N_{k}})$ converges to $\mu$ as $N_{k} \rightarrow \infty$. 

\subsubsection{Identification of limit} 

The PDE (in weak sense) associated with measure dependent Langevin dynamics is given by
\begin{align}\label{eld_pde}
    \la \phi(y), \mathcal{L}_{X(t)}(\de y) \ra -  \la \phi(y), \mathcal{L}_{X(0)}(\de y) \ra  & = -\int_{0}^{t} \la (\Sigma(\mathcal{L}_{X(s)})\nabla U(y) \cdot \nabla \phi(y)), \mathcal{L}_{X(s)}(\de y) \ra \de s 
    \nonumber \\ & \quad 
    + \int_{0}^{t}\frac{2}{\beta  } \la \trace(\nabla^2 \phi(y) \Sigma(\mathcal{L}_{X(s)})), \mathcal{L}_{X(s)}(\de y)\ra \de s, 
\end{align}
where $\phi \in C_{c}^{2}(\mathbb{R}^{d})$. 

Introduce the following functional on $\mathcal{P} (C([0,T];\mathbb{R}^{d}))$ for $t \in [0,T]$ and $\phi \in C^{2}_{c}(\mathbb{R}^{d})$ as
\begin{align}
    \Psi^{\phi}_{t}(\nu) :& =   \la \phi(y_{t}), \nu(\de y) \ra -  \la \phi(y_{0}), \nu(\de y) \ra  +\int_{0}^{t} \la (\Sigma(\nu(s))\nabla U(y_{s}) \cdot \nabla \phi(y_{s})), \nu(\de y) \ra \de s 
    \nonumber \\ & \quad 
    - \int_{0}^{t}\frac{2}{\beta  } \la (\trace(\text{Hess}^{\phi}(y_{s}) \Sigma(\nu_{s})), \nu_{s} \ra \de s  \label{knld_eqn_3.18}
\end{align}
We have already established the convergence of a sub-sequence of $(\mathcal{L}(\mathcal{E}^{N}))_{N\in \mathbb{N}}$ to a random measure $\mu \in \mathcal{P}(\mathcal{P} ( C([0,T]; \mathbb{R}^{d})))$. The next thing to show is that this random measure is actually nothing but $\mu := \delta_{\mathcal{L}_{X}}$. 
To this end, we establish the following bound for all $\phi\in C^{2}_{c}(\mathbb{R}^{d})$:
\begin{align}
    \mathbb{E}|\Psi^{\phi}_{t}(\mathcal{E}^{N})|^{2} \leq \frac{C}{N}, \label{elo_eqn_4.14}
\end{align}
where $C >0$ is independent of $N$.

Using Ito's formula, we have
\begin{align}
    \phi&(X^{i,N}(t))  = \phi(X^{i,N}(0)) - \int_{0}^{t} \la\nabla \phi(X^{i,N}(s)) \cdot \Sigma(\mathcal{E}^{N}(s)) \nabla U(X^{i,N}(s))\ra\de s  
    \nonumber \\ & \quad 
    + \int_{0}^{t} \frac{2}{\beta  }  \trace(\nabla^2\phi(X^{i,N}(s)) Q^{N}(s)(Q^{N})^{\top}(s))\de s
    + \int_{0}^{t} \sqrt{\frac{2 }{\beta  }} \la\nabla \phi(X^{i,N}(s)) \cdot Q^{N}(s) \de B^{i}(s)\ra.
\end{align}
This implies, using \eqref{knld_eqn_3.18}, that
\begin{align}
    \Psi^{\phi}_{t}(\mathcal{E}^{N}) =  \frac{1}{N}\sum_{i= 1}^{N}\int_{0}^{t} \sqrt{\frac{2 }{\beta  }} \la\nabla \phi(X^{i,N}(s)) \cdot Q^{N}(s) \de B^{i}(s)\ra
\end{align}
which on using the martingale property of Ito's integral results in the following:
\begin{align}
    \mathbb{E}|\Psi^{\phi}_{t}(\mathcal{E}^{N})|^{2} = \frac{1}{N^{2}}\sum_{i=1}^{N}\mathbb{E}\bigg| \int_{0}^{t} \sqrt{\frac{2 }{\beta  }} \la\nabla \phi(X^{i,N}(s)) \cdot Q^{N}(s) \de B^{i}(s)\ra\bigg|^{2}. 
\end{align}
Using Ito's isometry, we obtain
\begin{align}
    \mathbb{E}|\Psi^{\phi}_{t}(\mathcal{E}^{N})|^{2} = \frac{1}{N^{2}}\frac{2}{\beta}\sum_{i=1}^{N}\int_{0}^{t}\mathbb{E}\la\nabla \phi(X^{i,N}(s))\cdot \Sigma(\mathcal{E}^{N}(s))\nabla \phi(X^{i,N}(s))\ra^2 \de s
\end{align}
which on using (\ref{elo_cov_bound}) gives
\begin{align}\label{elo_eq_4.19}
     \mathbb{E}|\Psi^{\phi}_{t}(\mathcal{E}^{N})|^{2} \leq \frac{C}{N},
\end{align}
and thus establishing (\ref{elo_eqn_4.14}).

Note that for any bounded continuous function $f$, we have
\begin{align}
    \int_{\mathbb{R}^{d}} f(x) \mathcal{E}^{N}(t)(\de x) \rightarrow  \int_{\mathbb{R}^{d}} f(x) \mu_t(\de x)\quad \text{as}\quad N\rightarrow \infty,
\end{align}
and also the family of random variables $\{ \Sigma(\mathcal{E}^{N}(t)) \}_{N \in \mathbb{N}}$ is uniformly integrable due to (\ref{elo_cov_bound}). Consequently, we have (see, e.g., \cite[Thm.~2.20]{van2000asymptotic})
\begin{align}
    \Sigma(\mathcal{E}^{N}(t)) \rightarrow \Sigma(\mu_t) \quad \text{as}\quad N\rightarrow \infty. 
\end{align}

This ensures that $\Psi^{\phi}_{t}(\nu)$ is continuous function of $\nu$. Also, from (\ref{elo_eq_4.19}) it is clear that the family of random variables $\{ |\Psi^{\phi}_{t}(\mathcal{E}^{N}(t))|^{2}\}_{N\in \mathbb{N}}$ is uniformly integrable. Hence, due to Fatau's lemma, we have
\begin{align}
 \mathbb{E}|\Psi^{\phi}_{t}(\mu)|^{2}  =  
 \mathbb{E}\lim\limits_{N\rightarrow \infty} |\Psi^{\phi}_{t}(\mathcal{E}^{N})|^{2} = 
 \lim\limits_{N\rightarrow \infty}\mathbb{E}|\Psi^{\phi}_{t}(\mathcal{E}^{N})|^{2} = 0. 
\end{align}
This implies $\mu$ satisfies PDE (\ref{eld_pde}) (in weak sense)
\begin{align}
    \Psi^{\phi}_{t}(\mu) = 0,\quad \text{a.s.}
\end{align}

\subsubsection{Uniqueness of limit} 
The uniqueness of the solution (in weak sense) of PDE (\ref{eld_pde}) follows from the pathwise uniqueness of the solution of (\ref{el_sde_mf}) \cite{vaes2024sharpchaos}. 
This implies any arbitrary subsequence of $\{\mathcal{E}^{N}(t)\}_{N\in \mathbb{N}}$ has a convergent subsequence and all these subsequences converge to the same limit. Hence this completes the proof noticing that sequence itself is converging.

\section{Time discretization analysis}\label{sec:time_disc}

The implementation of (\ref{non_langevin_sde_weak_ips}) and (\ref{non_langevin_sde_strong_ips}) requires numerical discretization. Let the final time $T$ be fixed. We uniformly partition $[0,T]$ into $n$ sub-intervals of size $h$, i.e., $t_0=0<\dots<t_{n}=T$, $h = t_{k+1}-t_{k} = T/n$. Let $\alpha $ be a constant belonging to interval $ (0, 1/2)$. We denote the approximating Markov chain as $(X_{k}^{i,N})_{k=0}^{n}$ starting from $X_{0}^{i}$ for all $i =1,\dots, N$. We also denote $\beta_{k} := \beta(t_{k})$ and  $\C_{k}^{N} := \C(\mathcal{E}_{k}^{N})$ where
\begin{align}
\mathcal{E}_{k}^{N} = \frac{1}{N}\sum_{k=1}^{N}\delta_{X^{i,N}_{k}},
\end{align}
and therefore 
\begin{align}
    \C(\mathcal{E}_{k}^{N}) = \frac{1}{N}Q_{k}^{N} (Q_{k}^{N})^{\top} 
\end{align}
with $
    Q_{k}^{N} = [ X^{1,N}_{k}- \M_{k},\dots,  X^{N,N}_{k}- \M_{k}],\quad
    \M_{k} = \frac{1}{N}\sum_{i=1}^{N}X^{i,N}_{k}$.
As is known in the stochastic numerics literature (see \cite{hutzenthaler2011strong}) that Euler-Maruyama scheme may diverge in strong sense for coefficients with growth higher than linear. There are explicit schemes proposed to deal with the issue of unbounded moments arising due to non-linear growth \cite{hutzenthaler_tamed_2012, tretyakov2013fundamental}. Here, following \cite{hutzenthaler_tamed_2012},     we present tamed schemes to deal with non-linear growth of coefficients. We present two versions of tamed Euler-Maruyama scheme:
\begin{itemize}
    \item[(i)]  Particle-wise tamed Euler-Maruyama scheme : 
    \begin{equation}
    X^{i,N}_{k+1} = X^{i,N}_{k} -  \mathbb{B}_1(h, X^{i,N}_{k})\C_k^{N}\nabla U(X^{i,N}_k)h + \sqrt{\frac{2}{\beta N}} Q_{k}^{N}\xi_{k +1}^i h^{\frac{1}{2}},\quad k=0,\dots,n-1,
    \label{particle_tame_euler_lang}
\end{equation}
with
\begin{align*}
    \mathbb{B}_1(h, X^{i,N}_{k}) =  \diag\bigg(\frac{1}{1 + h^{\alpha}|\C_{k}^{N}\nabla U(X^{i,N}_{k})| }, \dots,\frac{1}{1 + h^{\alpha}|\C_{k}^{N}\nabla U(X^{i,N}_{k})| }\bigg), 
\end{align*}
where $|\C_{k}^{N}\nabla U(X^{i,N}_{k})|$ is the Euclidean norm and $\xi^{i}_{k+1}$ is standard normal $N$-dimensional random vector for all $i=1,\dots,N$ and $k=0,\dots,n-1$.

\item[(ii)]  Coordinate-wise tamed Euler-Maruyama scheme :

\begin{equation}
    X^{i,N}_{k+1} = X^{i,N}_{k} -  \mathbb{B}_2(h, X^{i,N}_{k}) \C_k^{N}\nabla U(X^{i,N}_k)h + \sqrt{\frac{2}{\beta_k N}} Q_{k}^{N}\xi_{k +1}^i h^{\frac{1}{2}},\quad k=0,\dots,n-1,
    \label{particle_tame_euler_lang2}
\end{equation}
with
\begin{align*}
    \mathbb{B}_2(h, X^{i,N}_{k}) =  \diag\bigg(\frac{1}{1 + h^{\alpha}|(\C_k^{N}\nabla U(X^{i,N}_k))_{1}|}, \dots, \frac{1}{1 +  h^{ \alpha}|(\C_k^{N}\nabla U(X^{i,N}_k))_{d}|}\bigg), 
\end{align*}
where 
$(\C_k^{N}\nabla U(X^{i,N}_k))_{j}$ denotes $j-$th coordinate of $\C_k^{N}\nabla U(X^{i,N}_k)$.

\end{itemize}
In (\ref{particle_tame_euler_lang}), due to particle-wise taming, we have
\begin{align}
   \frac{1}{ |\mathbb{B}_1(h, X^{i,N}_{k}) \C_k^{N}\nabla U(X^{i,N}_k)|} = \frac{1}{|\C_k^{N}\nabla U(X^{i,N}_k)|} + h^{\alpha},
\end{align}
and, therefore the following bound holds:
\begin{align}
|\mathbb{B}_1(h, X^{i,N}_{k}) \C_k^{N}\nabla U(X^{i,N}_k)| \leq \min\bigg(\frac{1}{h^{\alpha}}, |\C_{k}^{N}\nabla U(X^{i,N}_k)|\bigg). \label{tame_bound_3.1}
\end{align}
In the similar manner, the following is true for coordinate-wise taming:
\begin{align}
    |(\mathbb{B}_2(h, X^{i,N}_{k}) \C_k^{N}\nabla U(X^{i,N}_k))_{j}| \leq \min\bigg(\frac{1}{h^{\alpha}}, |(\C_{k}^{N}\nabla U(X^{i,N}_k))_{j}|\bigg),\quad j=1,\dots,d. 
\end{align}
\begin{remark}[Balancing technique]
Writing the schemes with a preconditioning on the drift by a matrix $\mathbb{B}$, one can imagine that there can be other possible choices of balancing matrices $ \mathbb{B}$. This particular idea of balancing appears in \cite{milstein1998balanced} for stiff SDEs, and has been utilized in \cite{tretyakov2013fundamental} (see also \cite{milstein2004stochastic}) to design balanced methods to deal with non-globally Lipschitz coefficients. In this context, taming can be considered as one particular choice of balancing matrix. 
\end{remark}

In this section, we aim to prove convergence of the tamed numerical scheme (\ref{particle_tame_euler_lang}) to its continuous limit as discretization step $h \rightarrow 0$. The techniques employed to obtain this convergence can also be used to get the convergence of (\ref{particle_tame_euler_lang2}). Here, however, we will only focus on (\ref{particle_tame_euler_lang}).

The main issue is to obtain the convergence uniform in $N$, where $N$ represents number of particles, considering that we have a non-globally Lipschitz drift coefficient. To this end, 
we let $\chi_{h}(t) = t_{k}$ for all $t_{k} \leq t < t_{k+1}$ and write the continuous time version of tamed Euler scheme (\ref{particle_tame_euler_lang}) as follows:
\begin{align} \label{cont_baleuler_ips}
\de Y^{i,N}(t) = F(Y^{i,N}(\chi_{h}(t)), \Sigma^{N}_{Y}(\chi_{h}(t))) \de t + \sqrt{\frac{2}{\beta}}Q^{N}_{Y}(\chi_{h}(t)) \de B^{i}(t),
\end{align}
where, for brevity, we have denoted
\begin{align}
    F(Y^{i,N}(\chi_{h}(t)), \Sigma^{N}(\chi_{h}(t))) = - \frac{ \Sigma^N_{Y} (\chi_{h}(t))\nabla U(Y^{i,N}(\chi_{h}(t)))}{ 1 + h^{\alpha }  |\Sigma^N_{Y} (\chi_{h}(t)) \nabla U(Y^{i,N}(\chi_{h}(t))) | } \label{ildo_eqn_3.25}
\end{align}
with $\alpha \in (0, 1/2)$. Also, we have denoted
\begin{align}
    \Sigma^{N}_{Y}(\chi_h(t)) := \frac{1}{N} Q^{N}_{Y}(\chi_{h}(t)) Q^{N}_{Y}(\chi_{h}(t))^{\top},
\end{align}
where
\[
Q^{N}_{Y}(t) = [Y^{1,N}(\chi_{h}(t))-\M^N_Y(\chi_{h}(t)),...,Y^{N,N}(t)-\M^N_Y(\chi_{h}(t))]
\]
with $\M^N_Y(\chi_{h}(t))$ being the ensemble mean given by
\begin{equation}
 \M^N_Y(\chi_{h}(t)):=   \frac{1}{N}\sum_{j=1}^N Y^{j,N}(\chi_{h}(t)).   
\end{equation}
In this section, we will need to strengthen Assumption~\ref{second_assum_on_U} as follows:
\begin{assumption}\label{third_assum_on_U}
  Let Assumption~\ref{second_assum_on_U} hold. Moreover, let  $U \in C^4(\mathbb{R}^d)$ and assume that its third and fourth order partial derivatives  are uniformly bounded in $x \in \R^d$. 
\end{assumption}
The main result of this section is the following theorem. 
\begin{theorem}\label{thm:convergence_time_disc}
Let Assumption~\ref{third_assum_on_U} hold.  Let $\sup_{1 \leq i \leq N} \E | X^{i,N}(0)|^{2} < \infty $, $\sup_{1 \leq i \leq N} \E | Y^{i,N}(0)|^{2} < \infty $ and
$|X^{i,N}(0) - Y^{i,N}(0)| \leq Ch^{1/2} $ with $C>0$ being independent of $N$ and $h$. Then, for all $t \in [0,T]$
\begin{align}
  \lim_{h \rightarrow 0} \sup_{i= 1,\dots, N} (\E| X^{i,N}(t) - Y^{i,N}(t)|^{2})^{1/2} = 0. 
\end{align} 
\end{theorem}

The first step towards establishing the convergence result in the above theorem is to obtain moment bounds that are independent of $N$ and $h$.
\begin{lemma}\label{mom_bo_dis_part_ildo}
   Let Assumption~\ref{third_assum_on_U} be satisfied. Let $p\geq 1$ be a constant. Let $  \sup_{1 \leq i \leq N} \E |Y^{i,N}(0)|^{2p} \leq C$ with $C>0$ being independent of $N$ and $h$. Then, the following holds: 
\begin{align}
    \sup_{1 \leq i \leq N} \E |Y^{i,N}(t)|^{2p} \leq C, \label{num_mom_bound}
\end{align}
where $C >0$ is independent of $h$ as well as $N$. 
\end{lemma}

In the below proof, we use Gr\"{o}nwall type arguments to obtain 
\begin{equation}
    \sup_{1 \leq i \leq N} \E \sup_{t\in[0,T]}U^{p}(Y^{i,N}(t)) \leq C,
\end{equation}
which, thanks to Assumption~\ref{third_assum_on_U}, provides the required moment bound (\ref{num_mom_bound}).

\begin{proof}
Applying Ito's formula on $U^{p}(Y^{i,N}(t))$, we get
\begin{align}
    \de U^{p}(Y^{i,N}(t)) & =  p U^{p -1}(Y^{i,N}(t))\la\nabla U (Y^{i,N}(t))\cdot F(Y^{i,N}(\chi_{h}(t)),  \Sigma^{N}_{Y}(\chi_{h}(t)))\ra \de t 
    \nonumber \\ &  \quad 
    + p(p-1)\frac{1}{\beta}U^{p-2}(Y^{i,N}(t)) \trace\big(\nabla U(Y^{i,N}(t))\nabla U ^{\top}(Y^{i,N}(t)) \Sigma^{N}_{Y}(\chi_{h}(t))\big) \de t
     \nonumber \\ &  \quad 
     + p \frac{1}{\beta}U^{p-1}(Y^{i,N}(t)) \trace\big( \nabla^2 U( Y^{i,N}(t)) \Sigma^{N}_{Y}(\chi_{h}(t))\big) \de t
      \nonumber \\ &  \quad  
      + p\sqrt{\frac{2}{\beta N}}U^{p-1}(Y^{i,N}(t))\la\nabla U(Y^{i,N}(t)) \cdot Q^{N}_{Y} (\chi_{h}(t)) \de B^{i}(t)\ra. \label{eldo_eqn_4.37}
\end{align}
For the sake of convenience, we denote
\begin{align}
    G(Y^{i,N}(t)) := \nabla U (Y^{i,N}(t)).
\end{align}
Applying Ito's formula on $j$-th component of $G$, we arrive at
\begin{align} \label{ildo_eqn_3.29}
G_{j}(Y^{i,N}(t)) &= G_{j}(Y^{i,N}(\chi_{h}(t))) + \int_{\chi_{h}(t)}^{t}\la \nabla G_{j} (Y^{i,N}(s)) \cdot F(Y^{i,N}(\chi_{h}(s)),  \Sigma^{N}_{Y}(\chi_{h}(s)))\ra \de s 
\nonumber \\ &  \quad     
 + \int_{\chi_{h}(t)}^{t}\frac{1}{\beta} \trace\big(\nabla^2 G_{j}( Y^{i,N}(s)) \Sigma^{N}_{Y}(\chi_{h}(s))\big) \de s 
 \nonumber \\ &  \quad  
 + \int_{\chi_{h}(t)}^{t}\sqrt{\frac{2}{\beta N}}\big\la \nabla G_{j}(Y^{i,N}(s)) \cdot Q^{N}_{Y}(\chi_{h}(s)) \de B^{i}(s)\big\ra.
\end{align}
We will estimate bounds on terms in (\ref{eldo_eqn_4.37}) one by one. Let us start with the first term in (\ref{eldo_eqn_4.37}):
\begin{align}
    A_1 :&=  \la\nabla U (Y^{i,N}(t))\cdot F(Y^{i,N}(\chi_{h}(t)),  \Sigma^{N}_{Y}(\chi_{h}(t)))\ra
    \nonumber \\ &  
    = \big\la\big(\nabla U (Y^{i,N}(t)) - \nabla U (Y^{i,N}(\chi_{h}(t)))\big)\cdot   F(Y^{i,N}(\chi_{h}(t)),  \Sigma^{N}_{Y}(\chi_{h}(t))) \big\ra 
    \nonumber \\ &  \quad  
    + \big\la\big(\nabla U (Y^{i,N}(\chi_{h}(t)))\big)\cdot   F(Y^{i,N}(\chi_{h}(t)),  \Sigma^{N}_{Y}(\chi_{h}(t))) \big\ra.
\end{align}
Note that
\begin{align}
     \big\la\big(\nabla U (Y^{i,N}(\chi_{h}(t)))\big)\cdot   F(Y^{i,N}(\chi_{h}(t)),  \Sigma^{N}_{Y}(\chi_{h}(t))) \big\ra\leq 0
\end{align}
due to the fact that $\Sigma^{N}(\chi_{h}(t))$ is ensemble covariance and hence positive semi-definite matrix (see (\ref{ildo_eqn_3.25})).  
We will now utilize (\ref{ildo_eqn_3.29}) to get 
\begin{align}
    A_1 & \leq \bigg|\sum_{j=1}^{d}F_{j}\bigg( \int_{\chi_{h}(t)}^{t}\la \nabla G_{j} (Y^{i,N}(s)) \cdot F(Y^{i,N}(\chi_{h}(s)),  \Sigma^{N}_{Y}(\chi_{h}(s)))\ra \de s \bigg)\bigg|
   \nonumber \\ &   \quad 
    + \bigg|\sum_{j=1}^{d}F_{j} \int_{\chi_{h}(t)}^{t}\frac{1}{\beta}\trace\big( \nabla^2 G_{j}( Y^{i,N}(s)) \Sigma^{N}_{Y}(\chi_{h}(s))\big) \de s\bigg|
    \nonumber \\ &   \quad 
    + \bigg|\sum_{j=1}^{d}F_{j} \int_{\chi_{h}(t)}^{t}\sqrt{\frac{2}{\beta N}}\big\la \nabla G_{j}(Y^{i,N}(s)) \cdot Q^{N}_{Y}(\chi_{h}(s)) \de B^{i}(s)\big\ra\bigg|, \label{eldo_eqn_4.42}
\end{align}
where $F_{j}$ denotes the $j$-th component of $F(Y^{i,N}(\chi_{h}(t)), \C^{N}_{Y}(\chi_{h}(t)))$.

Using (\ref{tame_bound_3.1}) and the fact that Frobenius norm of Hessian of $U$ is bounded, we obtain
\begin{align}
    \bigg|\sum_{j=1}^{d}& F_{j}\bigg( \int_{\chi_{h}(t)}^{t}\la \nabla G_{j} (Y^{i,N}(s)) \cdot F(Y^{i,N}(\chi_{h}(s)),  \Sigma^{N}_{Y}(\chi_{h}(s)))\ra \de s \bigg)\bigg| 
   \nonumber \\   & 
   \leq  \frac{1}{h^{\alpha}} \sum_{j=1}^{d}\bigg|\bigg( \int_{\chi_{h}(t)}^{t}\la \nabla G_{j} (Y^{i,N}(s)) \cdot F(Y^{i,N}(\chi_{h}(s)),  \Sigma^{N}_{Y}(\chi_{h}(s)))\ra \de s \bigg)\bigg|
    \nonumber \\   & 
   \leq  \frac{1}{h^{2\alpha}}\int_{\chi_{h}(t)}^{t}\sum_{j=1}^{d}\big| \nabla G_{j}(Y^{i,N}(s))\big| \de s  \leq  C h^{1 - 2\alpha}, \quad \label{eldo_eqn_4.43}
\end{align}
where $C>0$ is independent of $h$ and $N$. 

Dealing with the second term on right hand side of (\ref{eldo_eqn_4.42}) results in
\begin{align}
     \bigg|\sum_{j=1}^{d}&F_{j} \int_{\chi_{h}(t)}^{t}\frac{1}{\beta} \trace\big(\nabla^2 G_{j}( Y^{i,N}(s)) \Sigma^{N}_{Y}(\chi_{h}(s))\big) \de s\bigg| 
     \nonumber \\   &  
     \leq  C\frac{1}{h^{\alpha}} \int_{\chi_{h}(t)}^{t}\sum_{j=1}^{d} | \nabla^2 G_{j}(Y^{i,N}(s))| | \Sigma^{N}_{Y}(\chi_{h}(s))| \de s
     \leq  C\frac{1}{h^{\alpha}} \int_{\chi_{h}(t)}^{t} | \Sigma^{N}_{Y}(\chi_{h}(s))| \de s
       \nonumber \\   &   
       \leq C h^{1- \alpha} | \Sigma^{N}_{Y}(\chi_{h}(t))| 
\end{align}
since the third and fourth derivatives of $U$ are bounded. Note that
\begin{align}\label{eq:SigmaYBound}
   | \Sigma^{N}_{Y}(t)| \leq  C\Big(1 + \frac{1}{N}\sum_{i=1}^{N}|Y^{i,N}(t)|^{2}\Big)
   \leq  C\Big(1 + \frac{1}{N}\sum_{i=1}^{N}U(Y^{i,N}(t))\Big)
\end{align}
since $U $ has at least quadratic growth outside a compact set $\mathcal{K}$ (see Assumption~\ref{second_assum_on_U}). 
Consequently, we obtain
\begin{align}
     \bigg|\sum_{j=1}^{d}F_{j} \int_{\chi_{h}(t)}^{t}\frac{1}{\beta} \trace\big(\nabla^2 G_{j}( Y^{i,N}(s)) \Sigma^{N}_{Y}(\chi_{h}(s))\big) \de s\bigg|    
     \leq C h^{1 - \alpha} \Big(1 + \frac{1}{N}\sum_{i=1}^{N}U(X^{i,N}(\chi_{h }(t)))\Big).  \label{eldo_eqn_4.46}
\end{align}

In the similar manner, we have the following bound for the third term on the right hand side of (\ref{eldo_eqn_4.42}):
\begin{align}
    \bigg|\sum_{j=1}^{d}F_{j}& \int_{\chi_{h}(t)}^{t}\sqrt{\frac{2}{\beta}}\big\la \nabla G_{j}(Y^{i,N}(s)) \cdot Q^{N}_{Y}(\chi_{h}(s)) \de B^{i}(s)\big\ra\bigg|  
   \nonumber \\ & 
   \leq  \frac{1}{h^{\alpha}} \sum_{j=1}^{d}  \bigg|\int_{\chi_{h}(t)}^{t}\sqrt{\frac{2}{\beta N}}\big\la \nabla G_{j}(Y^{i,N}(s)) \cdot Q^{N}_{Y}(\chi_{h}(s)) \de B^{i}(s)\big\ra\bigg|. \label{eqn_eldo_4.47}
\end{align}
Combining (\ref{eldo_eqn_4.43}), (\ref{eldo_eqn_4.46}) and (\ref{eqn_eldo_4.47}) yields
\begin{align}
   A_{1} &\leq  C  h^{1 - 2\alpha} + C h^{1 - \alpha} \Big(1 + \frac{1}{N}\sum_{i=1}^{N}U(Y^{i,N}(\chi_{h }(t)))\Big) 
    \nonumber \\  &  \quad
    +  \frac{1}{h^{\alpha}} \sum_{j=1}^{d}  \bigg|\int_{\chi_{h}(t)}^{t}\sqrt{\frac{2}{\beta N}}\big\la \nabla G_{j}(Y^{i,N}(s)) \cdot Q^{N}_{Y}(\chi_{h}(s)) \de B^{i}(s)\big\ra\bigg|. \label{eqn_eldo_4.48}
\end{align}

To estimate a bound on second term of (\ref{eldo_eqn_4.37}), we utilize Cauchy-Bunyakovsky-Schwarz inequality to obtain
\begin{align}
    A_{2} :&= U^{p-2}(Y^{i,N}(t)) \trace\big(\nabla U(Y^{i,N}(t))\nabla U ^{\top}(Y^{i,N}(t))  \Sigma^{N}_{Y}(\chi_{h}(t))\big) 
    \nonumber \\ & 
    \leq C U^{p-2}(Y^{i,N}(t)) |\nabla U(Y^{i,N}(t))|^{2} |\Sigma^{N}_{Y}(\chi_{h}(t))|  \nonumber \\ & 
    \leq C(1 + U^{p-1}(Y^{i,N}(t))) \bigg( 1+ \frac{1}{N}\sum_{i=1}^{N}U(Y^{i,N}(\chi_{h}(t)))\bigg), \label{eqn_eldo_4.49}  
\end{align}
where we use Assumption~\ref{second_assum_on_U} and \eqref{eq:SigmaYBound} in the last step.

Analogously, for the third term of \eqref{eldo_eqn_4.37}, application of Cauchy-Bunyakovsky-Schwarz inequality yields
\begin{align}
     A_{3} :&= U^{p-1}(Y^{i,N}(t)) \trace\big( \nabla^2 U( Y^{i,N}(t)) \Sigma^{N}_{Y}(\chi_{h}(t))\big) 
     \nonumber \\ &
     \leq  C U^{p-1}(Y^{i,N}(t)) \bigg( 1+ \frac{1}{N}\sum_{i=1}^{N}U(Y^{i,N}(\chi_{h}(t)))\bigg), \label{eqn_eldo_4.50}
 \end{align}
 where we have utilized the fact that the norm of Hessian of potential function is bounded due to Assumption~\ref{second_assum_on_U}.
 
 Substituting (\ref{eqn_eldo_4.48}), (\ref{eqn_eldo_4.49}) and (\ref{eqn_eldo_4.50}) in (\ref{eldo_eqn_4.37}), we obtain 
 \begin{align}
      U^{p}(Y^{i,N}(t)) & \leq  U^{p}(Y^{i,N}(0)) + C h^{1 - 2\alpha } \int_{0}^{t}U^{p-1}(Y^{i,N}(s))\de s 
      \nonumber \\  & \quad 
      + C h^{1-\alpha}\int_{0}^{t}U^{p-1}(Y^{i,N}(s)) \bigg( 1 + \frac{1}{N}\sum_{i=1}^{N}U(Y^{i,N}(\chi_{h}(s)))\bigg) ~ \de s 
       \nonumber \\  & \quad 
      +  p \int_{0}^{t}U^{p-1}(Y^{i,N}(s))\frac{1}{h^{\alpha}} \sum_{j=1}^{d}  \bigg|\int_{\chi_{h}(s)}^{s}\sqrt{\frac{2}{\beta}}\big\la \nabla G_{j}(Y^{i,N}(r)) \cdot Q^{N}_{Y}(\chi_{h}(r)) \de B^{i}(r)\big\ra\bigg|\de s
       \nonumber \\  & \quad 
       + C \int_{0}^{t} U^{p-1}(Y^{i,N}(s))\bigg( 1 + \frac{1}{N}\sum_{i=1}^{N}U(Y^{i,N}(\chi_{h}(s)))\bigg)\de s
        \nonumber \\  & \quad 
        +  p\bigg|\int_{0}^{t}\sqrt{\frac{2}{\beta N}}U^{p-1}(Y^{i,N}(s))\la\nabla U(Y^{i,N}(s)) \cdot Q^{N}_{Y} (\chi_{h}(s)) \de B^{i}(s)\ra\bigg|. 
    \end{align}
Taking supremum over $[0,T]$ and then expectation on both sides, we ascertain
\begin{align*}
   \E  &\sup_{t \in[0,T]}  U^{p}(Y^{i,N}(t))  \leq  U^{p}(Y^{i,N}(0)) + C h^{1 - 2\alpha } \E    \int_{0}^{T} U^{p-1}(Y^{i,N}(s)) \de t 
   \\  & 
  +  C h^{1-\alpha}\E\int_{0}^{T} U^{p-1}(Y^{i,N}(s)) \bigg( 1 + \frac{1}{N}\sum_{i=1}^{N}U(Y^{i,N}(s))\bigg) ~ \de t
   \nonumber \\  & 
      +  p \E\int_{0}^{T}U^{p-1}(Y^{i,N}(s))\frac{1}{h^{\alpha}} \sum_{j=1}^{d}  \bigg|\int_{\chi_{h}(s)}^{s}\sqrt{\frac{2}{\beta}}\big\la \nabla G_{j}(Y^{i,N}(r)) \cdot \Sigma^{N}(\chi_{h}(r)) \de B^{i}(r)\big\ra\bigg|\de s
      \\ &
    +  \E\int_{0}^{T} U^{p-1}(Y^{i,N}(s)) \bigg( 1 + \frac{1}{N}\sum_{i=1}^{N}U(Y^{i,N}(s))\bigg) ~ \de t 
    \\ & 
   + p \E \sup_{t \in [0,T]}\bigg|\int_{0}^{t}\sqrt{\frac{2}{\beta N}}U^{p-1}(Y^{i,N}(s))\la\nabla U(Y^{i,N}(s)) \cdot Q^{N}_{Y} (\chi_{h}(s)) \de B^{i}(s)\ra\bigg|. \numberthis \label{eldo_eqn_3.43}
\end{align*}

We will first estimate the following term appearing on the right hand side of above inequality:
\begin{align*}
    D_{1} :&=\E\int_{0}^{T}U^{p-1}(Y^{i,N}(s))  \frac{1}{h^{\alpha}}\sum_{j=1}^{d}  \bigg|\int_{\chi_{h}(s)}^{s}\sqrt{\frac{2}{\beta }}\big\la \nabla G_{j}(Y^{i,N}(r)) \cdot Q^{N}_{Y}(\chi_{h}(r)) \de B^{i}(r)\big\ra\bigg|\de s
    \\ & 
    \leq  C\int_{0}^{T}\E \sup_{s\in [0,t]} U^{p}(Y^{i,N}(s)) \de t 
    \\  & \quad 
    + C \frac{1}{h^{\alpha p}} \int_{0}^{T}\sum_{j=1}^{d} \E\bigg| \int_{\chi_{h}(s)}^{s}\sqrt{\frac{2}{\beta N}}\big\la \nabla G_{j}(Y^{i,N}(r)) \cdot Q^{N}_{Y}(\chi_{h}(r)) \de B^{i}(r)\big\ra\bigg|^{p}\de s,
\end{align*}
where we have used generalized Young's inequality. Applying the Burkholder-Davis-Gundy inequality, and using boundedness of norm of Hessian of $U$ and \eqref{eq:SigmaYBound}, we obtain
\begin{align*}
    D_{1} &\leq C\int_{0}^{T}\E \sup_{s\in [0,t]} U^{p}(Y^{i,N}(s)) \de t
     +  \frac{C}{h^{\alpha p}} \int_{0}^{T}\sum_{j=1}^{d} \E \bigg(\int_{\chi_{h}(s)}^{s} | \nabla G_{j}(Y^{i,N}(r))|^{2} |\Sigma^{N}_{Y}(\chi_{h}(r))| \de r~ \bigg)^{p/2} \de s
     \\ &
     \leq C\int_{0}^{T}\E \sup_{s\in [0,t]} U^{p}(Y^{i,N}(s)) \de t 
     + \frac{C}{h^{\alpha p}} \int_{0}^{T}\E \bigg(\int_{\chi_{h}(s)}^{s}  \Big( 1 + \frac{1}{N}\sum_{j=1}^{N}U(Y^{j,N}(\chi_{h}(r))) \Big)^{2}\de r~ \bigg)^{p/2} \de s
     \\ &
     \leq C\int_{0}^{T}\E \sup_{s\in [0,t]} U^{p}(Y^{i,N}(s)) \de t 
     + C h^{p( 1/2 - \alpha )} \int_{0}^{T} \E  \Big( 1 + \frac{1}{N}\sum_{j=1}^{N}U(Y^{j,N}(\chi_{h}(s))) \bigg)^{p} \de s
     \\ &
     \leq C\int_{0}^{T}\E \sup_{s\in [0,t]} U^{p}(Y^{i,N}(s)) \de t
     + C h^{p( 1/2 - \alpha )} \int_{0}^{T} \E \sup_{s \in[0,t]} \Big( 1 + \frac{1}{N}\sum_{j=1}^{N}U^{p}(Y^{j,N}(s)) \bigg) \de t. \numberthis  \label{eldo_eqn_3.44}
\end{align*}
The last term remaining to be dealt with is \begin{align*}
    D_{2} := \E \sup_{t \in [0,T]}\bigg|\int_{0}^{t}\sqrt{\frac{2}{\beta N}}U^{p-1}(Y^{i,N}(s))\la\nabla U(Y^{i,N}(s)) \cdot Q^{N}_{Y} (\chi_{h}(s)) \de B^{i}(s)\ra\bigg|.
\end{align*}
Again using the Burkholder-Davis-Gundy inequality and \eqref{eq:SigmaYBound}, we get
\begin{align*}
    D_{2} &\leq   \E \bigg(\int_{0}^{T} \frac{2}{\beta}U^{2p-2}(Y^{i,N}(s))|\nabla U(Y^{i,N}(s))|^{2}|\C_{Y}^{N}(\chi_{h}(s))| \de s\bigg)^{1/2}
    \\& 
   \leq   \E \Bigg(\sup_{s \in [0,T]} U^{p-1}(Y^{i,N}(s)) \bigg(\int_{0}^{T} \frac{2}{\beta } (1 +  U(Y^{i,N}(s))) \Big( 1 + \frac{1}{N}\sum_{j=1}^{N} U(Y^{j,N}(\chi_{h}(s)))\Big) \de s \bigg)^{1/2}
    \\  & 
   \leq    \frac{1}{2}\E \big(\sup_{s \in [0,T]} U^{p}(Y^{i,N}(s))\big)  + C\E\bigg(\int_{0}^{T} \frac{2}{\beta} (1 +  U(Y^{i,N}(s))) \Big( 1 + \frac{1}{N}\sum_{j=1}^{N} U(Y^{j,N}(\chi_{h}(s)))\Big) \de s \bigg)^{p/2},
\end{align*}
where we have used generalized Young's inequality. Using Hölder's inequality, we have
\begin{align*}
 D_{2} &\leq     \frac{1}{2}\E \big(\sup_{s \in [0,T]} U^{p}(Y^{i,N}(s))\big)  + C\E\bigg(\int_{0}^{T}(1 +  U^{p/2}(Y^{i,N}(s))) \Big( 1 + \frac{1}{N}\sum_{j=1}^{N} U(Y^{j,N}(\chi_{h}(s)))\Big)^{p/2} \de s \bigg)
 \\ &  
 \leq  \frac{1}{2}\E \big(\sup_{t \in [0,T]} U^{p}(Y^{i,N}(t))\big) + C\int_{0}^{T} \Big( 1 + \frac{1}{N}\sum_{j=1}^{N} \E \sup _{s \in [0,t]}U^{p}(Y^{j,N}(s)) \Big)\de t \numberthis \label{eldo_eqn_3.45}
\end{align*}
Plugging the estimates obtained in (\ref{eldo_eqn_3.44}) and (\ref{eldo_eqn_3.45}) into (\ref{eldo_eqn_3.43}), we get 
\begin{align*}
     \E  \sup_{t \in[0,T]}  U^{p}(Y^{i,N}(t))  &\leq  2U^{p}(Y^{i,N}(0)) + C\int_{0}^{T} \Big( 1 + \frac{1}{N}\sum_{j=1}^{N} \E \sup _{s \in [0,t]}U^{p}(Y^{j,N}(s)) \Big)\de t  
     \\ & + C\int_{0}^{T}\E \sup_{s\in [0,t]} U^{p}(Y^{i,N}(s)) \de t  
\end{align*}
since $h^{1-2\alpha} \leq 1$ due to our choice of $\alpha \in (0,1/2)$. We mention again for the convenience of the reader that $C $ is a positive constant independent of $h$ and $N$. Taking supremum over $i \in \{1,\dots, N\}$ gives
\begin{align*}
  \sup_{i=1,\dots,N}  \E  \sup_{t \in[0,T]}  U^{p}(Y^{i,N}(t))\leq  2U^{p}(Y^{i,N}(0)) + C\int_{0}^{T}\sup_{i=1,\dots,N}\E \sup_{s\in [0,t]} U^{p}(Y^{i,N}(s)) \de t  
\end{align*}
which on applying Grönwall's Lemma provides the desired result. 
\end{proof}

\begin{lemma}\label{cont_step_error_lemma}
Let Assumption~\ref{second_assum_on_U} be satisfied. Let $\sup_{1 \leq i \leq N} \E | X^{i,N}(0)|^{2} < \infty $ and $\sup_{1 \leq i \leq N} \E | Y^{i,N}(0)|^{2} < \infty $. Then the following bound holds:
\begin{align}
    \sup_{1 \leq i \leq N} \E | Y^{i,N}(t) - Y^{i,N}(\chi_{h}(t))|^{2} \leq Ch,
\end{align}
where $C >0$ is independent of $N$ and $h$.
\end{lemma}
\begin{proof}
With an application of triangle inequality and Hölder's inequality, we have
\begin{align}
| &Y^{i,N}(t) - Y^{i,N}(\chi_{h}(t))|^{2} \leq  C\bigg( \bigg| \int_{\chi_{h}(t)}^{t} F (Y^{i,N}(\chi_{h}(s)), \Sigma^{N}_{Y}(\chi_{h}(s))) \de s \bigg|^{2} 
\nonumber  \\  & + \bigg| \int_{\chi_{h}(t)}^{t} \sqrt{\frac{2}{\beta N}}Q^{N}_{Y}(\chi_{h}(s)) \de B^{i}(s)\bigg|^{2}\bigg)
\nonumber  \\  &  
\leq C\int_{\chi_{h}(t)}^{t} |1|^{2} \de s \int_{\chi_{h}(t)}^{t} |F (Y^{i,N}(\chi_{h}(s)), \Sigma^{N}_{Y}(\chi_{h}(s)))|^2  \de s  \nonumber   + \bigg| \int_{\chi_{h}(t)}^{t}\sqrt{\frac{2}{\beta N}} Q^{N}_{Y}(\chi_{h}(s)) \de B^{i}(s)\bigg|^{2} 
\nonumber \\  & 
\leq C h \int_{\chi_{h}(t)}^{t} |F (Y^{i,N}(\chi_{h}(s)), \Sigma^{N}_{Y}(\chi_{h}(s)))|^{2} \de s  
+ \bigg| \int_{\chi_{h}(t)}^{t}\sqrt{\frac{2}{\beta N}} Q^{N}_{Y}(\chi_{h}(s)) \de B^{i}(s)\bigg|^{2} \nonumber 
\end{align}
which on using (\ref{tame_bound_3.1}) gives
\begin{align}
    | Y^{i,N}(t) - Y^{i,N}(\chi_{h}(t))|^{2} \leq Ch^{2(1-\alpha)} + C \bigg| \int_{\chi_{h}(t)}^{t}\sqrt{\frac{2}{\beta N}} Q^{N}_{Y}(\chi_{h}(s)) \de B^{i}(s)\bigg|^{2}.
\end{align}
Taking expectation on both sides and applying Ito's isometry, we achieve the following bound:
\begin{align}
    \E| Y^{i,N}(t) - Y^{i,N}(\chi_{h}(t))|^{2} & = Ch^{2(1-  \alpha)} + C\E\int_{\chi_{h}(t)}^{t} \frac{1}{N} | Q^{N}_{Y}(\chi_{h}(s)) \big( Q^{N}_{Y}(\chi_{h}(s)) \big)^{\top}| \de s
    \\ & \nonumber
    \leq Ch + C\E\int_{\chi_{h}(t)}^{t} | \Sigma^{N}_{Y}(\chi_{h}(s))| \de s 
     \leq Ch,
\end{align}
where $C>0$ is independent of $N$ and $h$. Hence, the lemma is proved. 
\end{proof}

To obtain our main result, we employ a localization strategy based on stopping times. This is also employed in the interacting particle setting in  \cite{tretyakov_jumpcbo_2023} (for the propagation of chaos proof, as well as for for the proof of the Euler scheme) and \cite{vaes2024sharpchaos} (for the propagation of chaos proof). More specifically, let
\begin{align}
    \tau^{1}_{R} &= \inf \bigg\{ t \leq 0 \; ; \; \frac{1}{N}\sum_{i=1}^{N }|X^{i,N}(t)|^{2} \geq R\bigg\},  \\
     \tau^{2}_{R} &= \inf \bigg\{ t \leq 0 \; ; \; \frac{1}{N}\sum_{i=1}^{N }|Y^{i,N}(t)|^{2} \geq R\bigg\}, 
\end{align}
and $\tau_{R} = \tau^{1}_{R} \wedge \tau^{2}_{R}$.

\begin{lemma}\label{intermediate_lemma}
    Let Assumption~\ref{second_assum_on_U} hold. Then the following inequality is satisfied:
    \begin{align*}
    \mathbb{E} \int_{0}^{t \wedge\tau_{R}} & |\Sigma(\cE^{N}(s))\nabla U(X^{i,N}(s)) + F(Y^{i,N}(\chi_{h}(s))))|^2 \de s 
    \leq 
    C h^{2\alpha} 
    \\  & 
    +  C (1 +  R)^2 \E \int_{0}^{t }  \bigg(|X^{i,N}(s\wedge \tau_{R}) - Y^{i,N}(s\wedge \tau_{R})|^{2} 
    \\  &  + |Y^{i,N}(s\wedge \tau_{R}) - Y^{i,N}(\chi_{h}(s\wedge \tau_{R}))|^{2}
    +   \frac{1}{N}\sum_{j =1}^{N} |Y^{j,N}(s\wedge \tau_{R}) - Y^{j,N}(\chi_{h}(s\wedge \tau_{R}))|^{2}
    \\  & 
     + \frac{1}{N}\sum\limits_{j=1}^{N} |X^{j,N}(s\wedge \tau_{R}) - Y^{j,N}(s \wedge \tau_{R})|^{2}
    \bigg)\de s
    \end{align*}
    where $X^{i,N}(s)$ is from (\ref{non_langevin_sde_weak_ips}), $Y^{i,N}(s)$ is from  (\ref{cont_baleuler_ips}) and $C >0$ is independent of $h$, $N$ and $R$. 
\end{lemma}
\begin{proof}

    By repeated use of the triangle inequality, we can split the integrand as
    \begin{align*}
    & |\Sigma^{N}(s)\nabla U(X^{i,N}(s)) + F(Y^{i,N}(\chi_{h}(s)))|^2 
     \leq  |\Sigma^{N}(s)\nabla U(X^{i,N}(s)) - \Sigma^{N}_{Y}(s)\nabla U(Y^{i,N}(s))|^2 
     \\ & \quad 
     +   |\Sigma^{N}_{Y}(s)\nabla U(Y^{i,N}(s))  - \Sigma_{Y}^{N}(\chi_{h}(s))\nabla U(Y^{i,N}(\chi_{h}(s)))|^2
      \\ & \quad 
     + |\Sigma^{N}_{Y}(\chi_{h}(s))\nabla U(Y^{i,N}(\chi_{h}(s)))  + F(Y^{i,N}(\chi_{h}(s)))|^2
     \numberthis \label{ildo_eqn_3.53}
     \\  &  =: (\text{I}) + (\text{II}) + (\text{III}). \numberthis
\end{align*}
We will analyze the three terms separately, starting with (I). To this end, we first split the term (I)  further as follows:
    \begin{align*}
    \text{(I)}  :&=   |\Sigma^{N}(s)\nabla U(X^{i,N}(s)) - \Sigma^{N}_{Y}(s)\nabla U(Y^{i,N}(s)))|^2 
    \leq |\Sigma^{N}(s) (\nabla U(X^{i,N}(s)) - \nabla U(Y^{i,N}(s)))|^2 
    \\ & \quad 
    +|(\Sigma^{N}(s)  -  \Sigma^{N}_{Y}(s))\nabla U(Y^{i,N}(s))|^2
     \\ & 
    =: B_{1} + B_{2}. \numberthis \label{eldo_eqn_3.55}
    \end{align*}
We have the following bound for $B_1$ due to \eqref{elo_cov_bound} and the fact that $\nabla U $ is Lipschitz:
\begin{align*}
  B_{1} &:= |\Sigma^{N}(s) (\nabla U(X^{i,N}(s)) - \nabla U(Y^{i,N}(s)))|^2 
  \leq |\Sigma^{N}(s)|^2 |\nabla U(X^{i,N}(s)) - \nabla U(Y^{i,N}(s))|^2 
    \\ & \leq 
    C \Big[ 1 +\frac{1}{N}\sum_{j=1}^{N}|X^{j,N}(s)|^{2}\Big]^2  |X^{i,N}(s) - Y^{i,N}(s)|^{2}, \numberthis \label{eldo_eqn_4.64}
\end{align*}
where $C >0$ is a positive constant independent of $N$ and $h$.
Next, to bound $B_{2}$, note that 
\begin{align*}
    |Q^{N}(s) -  Q^{N}_{Y}(s)| & =  \Bigg(\sum_{i=1}^{N}\bigg| X^{i,N}(s) - Y^{i,N}(s) - \frac{1}{N}\sum_{j=1}^{N}(X^{j,N}(s) - Y^{j,N}(s))\bigg|^{2}\Bigg)^{1/2}
   \\  & 
   \leq \sqrt{2}\Bigg( \sum_{i=1}^{N}| X^{i,N}(s) - Y^{i,N}(s)|^{2} + N \bigg|\frac{1}{N}\sum_{j=1}^{N}(X^{j,N}(s) - Y^{j,N}(s))\bigg|^{2}\Bigg)^{1/2}.
\end{align*} 
This implies 
\begin{align*}
   \frac{1}{\sqrt{N}} |Q^{N}(s) -  Q^{N}_{Y}(s)| \leq 2 \Bigg( \frac{1}{N}\sum_{i=1}^{N}| X^{i,N}(s) - Y^{i,N}(s)|^{2}\Bigg)^{1/2}. \numberthis \label{3.56_ildo_eqn}
\end{align*}
In the similar manner, we have
\begin{align}
     \frac{1}{\sqrt{N}} |Q^{N}(s)| \leq 2 \bigg(\frac{1}{N}\sum\limits_{i=1}^{N}|X^{i,N}(s)|^{2}\bigg)^{1/2} \quad \text{and}\;\;
      \frac{1}{\sqrt{N}} |Q^{N}_{Y}(s)| \leq 2 \bigg(\frac{1}{N}\sum\limits_{i=1}^{N}|Y^{i,N}(s)|^{2}\bigg)^{1/2}. \label{3.58_ildo_eqn}
\end{align}
Using (\ref{3.56_ildo_eqn}) and (\ref{3.58_ildo_eqn}), we obtain
\begin{align*}
    |\Sigma^{N}(s)  -  \Sigma^{N}_{Y}(s)| & = \Big| \frac{1}{N} Q^{N}(s) (Q^{N}(s))^{\top}
    -  \frac{1}{N} Q^{N}_{Y}(s) (Q^{N}_{Y}(s))^{\top}\Big|
    \\  & 
  \leq  \Big| \frac{1}{N}  ( Q^{N}(s) - Q^{N}_{Y}(s)) (Q^{N}(s))^{\top}\Big| +  \Big|\frac{1}{N} Q^{N}_{Y}(s) ((Q^{N}(s))^{\top} - (Q^{N}_{Y}(s))^{\top})\Big|
 \\&  \leq C\bigg(\bigg(\frac{1}{N}\sum_{j=1}^{N}|X^{j,N}(s)|^{2}\bigg)^{1/2} + \bigg(\frac{1}{N}\sum_{j=1}^{N}|Y^{j,N}(s)|^{2}\bigg)^{1/2}\bigg)
 \\  & \quad \times \bigg( \frac{1}{N}\sum_{j =1}^{N} |X^{j,N}(s) - Y^{j,N}(s)|^{2}\bigg)^{1/2},
\end{align*}
where $C >0$ is a positive constant independent of $N$ and $h$. 
Using the above calculation we arrive at the following bound for $B_{2}$ in (\ref{eldo_eqn_3.55}):
\begin{align*}
    B_{2} :&=  |(\Sigma^{N}(s)  -  \Sigma^{N}_{Y}(s))\nabla U(Y^{i,N}(s))|^2
    \\ & 
     \leq C|\nabla U( Y^{i,N}(s))|^2 \bigg(\bigg(\frac{1}{N}\sum_{j=1}^{N}|X^{j,N}(s)|^{2}\bigg) + \bigg(\frac{1}{N}\sum_{j=1}^{N}|Y^{j,N}(s)|^{2}\bigg)\bigg)
    \\ &  \quad 
   \times  \bigg( \frac{1}{N}\sum_{j =1}^{N} |X^{j,N}(s) - Y^{j,N}(s)|^{2}\bigg) . \numberthis \label{eldo_eqn_4.65}
\end{align*}
Due to exchangeability of discretized particle system, we have for all $i =1,\dots, N$
\begin{align*}
    \mathbb{E}&\bigg(|\nabla U( Y^{i,N}(s\wedge \tau_R))|^{2}\bigg(\bigg(\frac{1}{N}\sum_{j=1}^{N}|X^{j,N}(s\wedge \tau_R)|^{2}\bigg) + \bigg(\frac{1}{N}\sum_{j=1}^{N}|Y^{j,N}(s\wedge \tau_R)|^{2}\bigg)\bigg) 
    \\  & \quad \times \bigg( \frac{1}{N}\sum_{j =1}^{N} |X^{j,N}(s\wedge \tau_{R}) - Y^{j,N}(s\wedge \tau_{R})|^{2}\bigg) \bigg) 
    \\  & 
    = \mathbb{E}\bigg(\frac{1}{N}\sum_{i=1}^{N}|\nabla U( Y^{i,N}(s\wedge \tau_{R}))|^{2}
    \bigg(\bigg(\frac{1}{N}\sum_{j=1}^{N}|X^{j,N}(s\wedge \tau_R)|^{2}\bigg) + \bigg(\frac{1}{N}\sum_{j=1}^{N}|Y^{j,N}(s\wedge \tau_R)|^{2}\bigg)\bigg) 
    \\  & \quad \times 
    \bigg( \frac{1}{N}\sum_{j =1}^{N} |X^{j,N}(s\wedge \tau_{R}) - Y^{j,N}(s\wedge \tau_{R})|^{2}\bigg) \bigg) .
\end{align*}

Using (\ref{eldo_eqn_4.64}) and (\ref{eldo_eqn_4.65}) yield
\begin{align*}
    \mathbb{E}&\int_{0}^{t \wedge \tau_R} (I) \de s \leq  C (1 + R)^2\int_{0}^{t}  \bigg( |X^{i,N}(s \wedge \tau_{R}) - Y^{i,N}(s\wedge \tau_R)|^{2} 
    \\  & \quad
    + \frac{1}{N}\sum\limits_{j=1}^{N} |X^{j,N}(s\wedge \tau_{R}) - Y^{j,N}(s \wedge \tau_{R})|^{2} \bigg) \de s \\  & 
     \leq  C (1 + R)^2\int_{0}^{t}  \E\bigg( |X^{i,N}(s \wedge \tau_{R}) - Y^{i,N}(s\wedge \tau_R)|^{2} 
    + \frac{1}{N}\sum\limits_{j=1}^{N} |X^{j,N}(s\wedge \tau_{R}) - Y^{j,N}(s \wedge \tau_{R})|^{2} \bigg) \de s \numberthis \label{eqn_ildo_3.57}
\end{align*}
In the similar manner, we have that
 \begin{align*}
    \text{(II)} :&=  |\Sigma_{Y}^{N}(s)\nabla U(Y^{i,N}(s))  - \Sigma^{N}_{Y}(\chi_{h}(s))\nabla U(Y^{i,N}(\chi_{h}(s)))|^2 
    \\&
    \leq |\Sigma^{N}_{Y}(s) (\nabla U(Y^{i,N}(s)) - \nabla U(Y^{i,N}(\chi_{h}(s))))|^2 
    \\ & \quad 
    +  |(\Sigma^{N}_{Y}(s)  -  \Sigma^{N}_{Y}(\chi_{h}(s)))\nabla U(Y^{i,N}(\chi_{h}(s)))|^2
     \\ & 
    =: E_{1} + E_{2}. \numberthis \label{eldo_eqn_4.63}
    \end{align*}
Applying the analogous arguments used in obtaining bounds for $B_1$ and $B_{2}$, we obtain
\begin{align*}
    E_{1} :&=  |\Sigma^{N}_{Y}(s) (\nabla U(Y^{i,N}(s)) - \nabla U(Y^{i,N}(\chi_{h}(s))))|^2 
    \\  & 
    \leq  C \bigg( 1 + \frac{1}{N}\sum_{j=1}^{N}|Y^{j,N}(s)|^{2}\bigg)^2
 \Big( |Y^{i,N}(s) - Y^{i,N}(\chi_{h}(s))|^{2}\Big),\numberthis \label{eldo_eqn_4.70_D1}
\end{align*}
and 
\begin{align*}
    E_{2} :&=  |(\Sigma^{N}_{Y}(s)  -  \Sigma^{N}_{Y}(\chi_{h}(s)))\nabla U(Y^{i,N}(\chi_{h}(s)))|^2
    \\ &
    \leq C\big(1 + |Y^{i,N}(\chi_{h}(s))|^2 \big)\bigg( 1 + \frac{1}{N}\sum_{j=1}^{N}|Y^{j,N}(s)|^{2} + \frac{1}{N}\sum_{j=1}^{N}|Y^{j,N}(\chi_{h}(s))|^{2}\bigg)
    \\  & \quad
   \times  \frac{1}{N}\sum_{j =1}^{N} |Y^{j,N}(s) - Y^{j,N}(\chi_{h}(s))|^{2}, \numberthis \label{eldo_eqn_4.71_D2}
\end{align*}
where $C >0$ is independent of $N$ and $h$. 

The bounds obtained in (\ref{eldo_eqn_4.70_D1}) and (\ref{eldo_eqn_4.71_D2}) imply that the following holds:
\begin{align*}
    \E &\int_{0}^{t \wedge \tau_{R}} (\text{II}) \de s \leq  C (1 +  R)^2 \E \int_{0}^{t }  \bigg(|Y^{i,N}(s\wedge \tau_{R}) - Y^{i,N}(\chi_{h}(s\wedge \tau_{R}))|^{2}
    \\  & +   \frac{1}{N}\sum_{j =1}^{N} |Y^{j,N}(s\wedge \tau_{R}) - Y^{j,N}(\chi_{h}(s\wedge \tau_{R}))|^{2} \bigg) \de s \numberthis \label{eqn_ildo_3.61}
\end{align*}
where we have used Young's inequality in last step.

We shift our attention to the remaining term, i.e., to term (III) in (\ref{ildo_eqn_3.53}), which is defined as
\begin{align*}
 \text{(III)} :&=    |\Sigma_{Y}(\chi_{h}(s))\nabla U(Y^{i,N}(\chi_{h}(s)))  + F(Y^{i,N}(\chi_{h}(s)))|^{2}.
\end{align*}
Note that
\begin{align*}
  |\Sigma^{N}_{Y}(\chi_{h}(s))& \nabla U(Y^{i,N}(\chi_{h}(s)))  + F(Y^{i,N}(\chi_{h}(s)))|^2  
  \\  & 
  = \bigg|\Sigma^{N}_{Y}(\chi_{h}(s))\nabla U(Y^{i,N}(\chi_{h}(s)))  - \frac{ \Sigma^N_{Y} (\chi_{h}(s))\nabla U(Y^{i,N}(\chi_{h}(s)))}{ 1 + h^{\alpha }  | \Sigma^N_{Y} (\chi_{h}(s))\nabla U(Y^{i,N}(\chi_{h}(s))) |}\bigg|^2
  \\  & 
  \leq  h^{2\alpha}\frac{| \Sigma^N_{Y} (\chi_{h}(s))\nabla U(Y^{i,N}(\chi_{h}(s)))|^{4}}{ (1 +  h^{\alpha }  |\Sigma^N_{Y} (\chi_{h}(s))\nabla U(Y^{i,N}(\chi_{h}(s)) )| )^2}.
\end{align*}
This results in 
\begin{align*}
     \E \int_{0}^{t \wedge \tau_{r}} (\text{III}) ds &\leq   C h^{2\alpha}\int_{0}^{t} \E \frac{| \Sigma^N_{Y} (\chi_{h}(s))\nabla U(Y^{i,N}(\chi_{h}(s)))|^{4}}{ (1 +  h^{\alpha }  |\Sigma^N_{Y} (\chi_{h}(s))\nabla U(Y^{i,N}(\chi_{h}(s)) )|)^2 } \de s 
     \\  &  
     \leq  Ch^{2\alpha}, \numberthis \label{eqn_ildo_3.62} 
\end{align*}
where we have used moment bounds from Lemma~\ref{mom_bo_dis_part_ildo}. Combining the estimate obtained in 
 (\ref{eqn_ildo_3.57}), (\ref{eqn_ildo_3.61}) and (\ref{eqn_ildo_3.62}), we ascertain
 \begin{align*}
    \mathbb{E}& \int_{0}^{t \wedge\tau_{R}} |\Sigma^{N}(s)\nabla U(X^{i,N}(s)) + F(Y^{i,N}(\chi_{h}(s)))|^2 \de s 
    \\  & \leq C h^{2\alpha} +  C (1 +  R)^2 \E \int_{0}^{t }  \bigg(|X^{i,N}(s\wedge \tau_{R}) - Y^{i,N}(s\wedge \tau_{R})|^{2} 
    \\  &  + |Y^{i,N}(s\wedge \tau_{R}) - Y^{i,N}(\chi_{h}(s\wedge \tau_{R}))|^{2}
    +   \frac{1}{N}\sum_{j =1}^{N} |Y^{j,N}(s\wedge \tau_{R}) - Y^{j,N}(\chi_{h}(s\wedge \tau_{R}))|^{2}
    \\  & 
     + \frac{1}{N}\sum\limits_{j=1}^{N} |X^{j,N}(s\wedge \tau_{R}) - Y^{j,N}(s \wedge \tau_{R})|^{2}
    \bigg)\de s.
    \end{align*}
This completes the proof. 
\end{proof}

\subsection{Proof of Theorem~\ref{thm:convergence_time_disc}}
\begin{proof}
We split the sample space as follows:
\begin{align}
    \E| X^{i,N}(t) - Y^{i,N}(t)|^{2} = \E(| X^{i,N}(t) - Y^{i,N}(t)|^{2} I_{\Omega_{R}}) + \E(| X^{i,N}(t) - Y^{i,N}(t)|^{2} I_{\Lambda_{R}}).
\end{align}
As is clear, we have the following bound:
\begin{align*}
     \E(| X^{i,N}(t) - Y^{i,N}(t)| I_{\Omega_{R}})^{2} \leq \E|X^{i,N}(t\wedge \tau_{R}) &- Y^{i,N}(t\wedge \tau_{R})|^{2}.
\end{align*}
Using Ito's formula, we have
\begin{align}
    |X^{i,N}(t\wedge \tau_{R}) &- Y^{i,N}(t\wedge \tau_{R})|^{2} = |X^{i,N}(0) - Y^{i,N}(0)|^{2} 
    \nonumber \\  & 
    - 2\int_{0}^{t\wedge \tau_{R}} \big\la(X^{i,N}(s) - Y^{i,N}(s)) \cdot (\Sigma(\cE^{N}(s))\nabla U(X^{i,N}(s)) + F(Y^{i,N}(\chi_{h}(s))))\big\ra \de s
     \nonumber \\  & 
     + \int_{0}^{t\wedge \tau_{R}} \frac{2}{\beta N}\text{trace}\big( [Q^{N}(s)
     - Q^{N}_{Y}(\chi_{h}(s))][Q^{N}(s)
     - Q^{N}_{Y}(\chi_{h}(s))]^{\top}\big) \de s
     \nonumber \\  &
      + 2\int_{0}^{t\wedge \tau_{R}}\sqrt{\frac{2}{\beta N}}\bigg\la(X^{i,N}(s) - Y^{i,N}(s)) \cdot  \Big(Q^{N}(s) - Q^{N}_{Y}(\chi_{h}(s)) \Big) \de  W^{i}(s)\bigg\ra. \label{eldo_eqn_4.58}
\end{align}
Due to Doob's optional stopping theorem, we have
\begin{align}
\E \bigg(\int_{0}^{t\wedge \tau_{R}}\sqrt{\frac{2}{\beta(s)}}\bigg\la(X^{i,N}(s) - Y^{i,N}(s)) \cdot  \Big(Q^{N}(s) - Q^{N}_{Y}(\chi_{h}(s)) \Big) \de W^{i}(s)\bigg\ra\bigg) = 0.  \label{eldo_eqn_3.69}
\end{align}
Note that the Doob's theorem can be applied due to the fact that $t \wedge \tau_{R}$ is bounded stopping time, and moments are bounded (see Lemma~\ref{mom_bo_dis_part_ildo}).  
We have already dealt with expectation of second term on right hand side of (\ref{eldo_eqn_4.58}) in Lemma~\ref{intermediate_lemma}.  For the third term on the right hand side of (\ref{eldo_eqn_4.58}), we utilize (\ref{3.56_ildo_eqn}) and (\ref{3.58_ildo_eqn}) to arrive at the following estimate:
\begin{align*}
    \E &\bigg(\int_{0}^{t\wedge \tau_{R}} \frac{2}{\beta N}\text{trace}\big( [Q^{N}(s)
     - Q^{N}_{Y}(\chi_{h}(s))][Q^{N}(s)
     - Q^{N}_{Y}(\chi_{h}(s))]^{\top}\big) \de s \bigg)
     \\  &
     \leq C\E \int_{0}^{t\wedge \tau_{R}} \frac{1}{N} \sum_{j=1}^{N} | X^{j,N}(s) -  Y^{j,N}(\chi_{h}(s))|^{2} \de s 
     \\  & 
      \leq C\E \int_{0}^{t\wedge \tau_{R}} \frac{1}{N} \sum_{j=1}^{N} | X^{j,N}(s) -  Y^{j,N}(s)|^{2} \de s  +   C\E \int_{0}^{t\wedge \tau_{R}} \frac{1}{N} \sum_{j=1}^{N} | Y^{j,N}(s) -  Y^{j,N}(\chi_{h}(s))|^{2} \de s 
     \\   &
     \leq C \E \int_{0}^{t} \frac{1}{N} \sum_{j=1}^{N} | X^{j,N}(s \wedge \tau_{R}) -  Y^{j,N}(s \wedge \tau_{R})|^{2} \de s
      \\  & \quad  +  C\E \int_{0}^{t} \frac{1}{N} \sum_{j=1}^{N} | Y^{j,N}(s \wedge \tau_{R}) -  Y^{j,N}(\chi_{h}(s \wedge \tau_{R}))|^{2} \de s, \numberthis \label{ildo_eqn_3.70}
\end{align*}
where $C >0$ is independent of $h$ and $N$. Combining the results of  (\ref{eldo_eqn_3.69}), (\ref{ildo_eqn_3.70}) and Lemma~\ref{intermediate_lemma}, we get
\begin{align*}
    \E  |X^{i,N}(t\wedge \tau_{R}) &- Y^{i,N}(t\wedge \tau_{R})|^{2}  \leq |X^{i,N}(0) - Y^{i,N}(0)|^{2} 
    \\  & 
    \leq C h^{2\alpha} +  C (1 +  R)^2 \E \int_{0}^{t }  \bigg(|X^{i,N}(s\wedge \tau_{R}) - Y^{i,N}(s\wedge \tau_{R})|^{2} 
    \\  &  + |Y^{i,N}(s\wedge \tau_{R}) - Y^{i,N}(\chi_{h}(s\wedge \tau_{R}))|^{2}
    +   \frac{1}{N}\sum_{j =1}^{N} |Y^{j,N}(s\wedge \tau_{R}) - Y^{j,N}(\chi_{h}(s\wedge \tau_{R}))|^{2}
    \\  & 
     + \frac{1}{N}\sum\limits_{j=1}^{N} |X^{j,N}(s\wedge \tau_{R}) - Y^{j,N}(s \wedge \tau_{R})|^{2}
    \bigg)\de s.
\end{align*}
Using Lemma~\ref{cont_step_error_lemma}, we ascertain
\begin{align*}
     \E  |X^{i,N}&(t\wedge \tau_{R}) - Y^{i,N}(t\wedge \tau_{R})|^{2}  \leq Ch + C(1+R)^2 h  
     \\  & 
     + C(1+ R)^2\int_{0}^{t} \bigg(\E|X^{i,N}(s\wedge \tau_{R}) - Y^{i,N}(s\wedge \tau_{R})|^{2} + \frac{1}{N}\sum\limits_{j=1}^{N} \E|X^{j,N}(s\wedge \tau_{R}) - Y^{j,N}(s \wedge \tau_{R})|^{2}\bigg) \de s,
\end{align*}
which on taking supremum over $j =1 ,\dots, N$, we get
\begin{align*}
     \sup_{i=1,\dots,N}\E  |X^{i,N}&(t\wedge \tau_{R}) - Y^{i,N}(t\wedge \tau_{R})|^{2}  \leq Ch^{2\alpha} + C(1+R)^2 h  
     \\  & 
     + C(1+ R)^2\int_{0}^{t} \sup_{i=1,\dots,N}\E|X^{i,N}(s\wedge \tau_{R}) - Y^{i,N}(s\wedge \tau_{R})|^{2}\de s.
\end{align*}
Applying Gronwall's lemma, we have
\begin{align*}
     \E  |X^{i,N}&(t\wedge \tau_{R}) - Y^{i,N}(t\wedge \tau_{R})|^{2} \leq  C ( h^{2\alpha} + (1+ R)^2 h) e^{C (1+ R)^2}.
\end{align*}
This implies 
\begin{align*}
     \E \big( |X^{i,N}(t) - Y^{i,N}(t)|^{2} I_{\Omega_{R}}\big) \leq  C ( h^{2\alpha} + (1+ R)^2 h) e^{C (1+ R)^2},
\end{align*}
where $C >0$ is independent of $h$ and $N$. 

Using H\"older's inequality, Lemma~\ref{mom_bo_dis_part_ildo} and Markov's inequality, we have
\begin{align*}
    \E \big( |X^{i,N}(t) - Y^{i,N}(t)|^{2} I_{\Lambda_{R}}\big) 
    & \leq C\big(\E  |X^{i,N}(t)|^{4} + \E | Y^{i,N}(t)|^{4}\big)^{1/2} \mathbb{P}(\Lambda_{R}(\omega))  \\  & 
    \leq  C \frac{ \frac{1}{N}\sum_{i=1}^{N}\mathbb{E}|X^{i,N}(t)|^{2}   }{ R}
    \leq \frac{C}{R},
 \end{align*}
 where $C >0 $ is independent of $h$, $N$ and $R$. With appropriate choice of $R $ depending on $h$, we get the desired result. 
\end{proof}

\section*{Acknowledgments}
This work was supported by the Wallenberg AI, Autonomous Systems and Software Program (WASP) funded by the Knut and Alice Wallenberg Foundation.

\bibliography{}

\newcommand{\etalchar}[1]{$^{#1}$}
\begin{thebibliography}{BSWW19}

\bibitem[ANO{\etalchar{+}}09]{aanonsen2009ensemble}
S.~I. Aanonsen, G.~N{\oe}vdal, D.~S. Oliver, A.~C. Reynolds, and B.~Vall{\`e}s.
\newblock The ensemble {K}alman filter in reservoir engineering—a review.
\newblock {\em Spe Journal}, 14(03):393--412, 2009.

\bibitem[BBCG08]{bakry2008simple}
D.~Bakry, F.~Barthe, P.~Cattiaux, and A.~Guillin.
\newblock A simple proof of the poincar{\'e} inequality for a large class of
  probability measures.
\newblock {\em Electronic Communications in Probability}, 2008.

\bibitem[Bil13]{billingsley2013convergence}
P.~Billingsley.
\newblock {\em Convergence of probability measures}.
\newblock John Wiley \& Sons, 2013.

\bibitem[BSW18]{blomker2018strongly}
D.~Blömker, C.~Schillings, and P.~Wacker.
\newblock A strongly convergent numerical scheme from ensemble {K}alman
  inversion.
\newblock {\em SIAM Journal on Numerical Analysis}, 56(4):2537--2562, 2018.

\bibitem[BSWW19]{blomker2019well}
D.~Bl{\"o}mker, C.~Schillings, P.~Wacker, and S.~Weissmann.
\newblock Well posedness and convergence analysis of the ensemble {K}alman
  inversion.
\newblock {\em Inverse Problems}, 35(8):085007, 2019.

\bibitem[CCTT18]{carrillo2018analytical}
J.~A. Carrillo, Y.-P. Choi, C.~Totzeck, and O.~Tse.
\newblock An analytical framework for consensus-based global optimization
  method.
\newblock {\em Mathematical Models and Methods in Applied Sciences},
  28(06):1037--1066, 2018.

\bibitem[CDR24]{chen_goncalo_2024euler}
X.~Chen and G.~Dos~Reis.
\newblock Euler simulation of interacting particle systems and
  {M}c{K}ean--{V}lasov {SDE}s with fully super-linear growth drifts in space
  and interaction.
\newblock {\em IMA Journal of Numerical Analysis}, 44(2):751--796, 2024.

\bibitem[CG22]{cattiaux_guillin2022functional_supplement}
P.~Cattiaux and A.~Guillin.
\newblock Supplement to functional inequalities for perturbed measures with
  applications to log-concave measures and to some {B}ayesian problems.
\newblock {\em Bernoulli}, 28(4):2294--2321, 2022.

\bibitem[CHS87]{chiang1987diffusion}
T.-S. Chiang, C.-R. Hwang, and S.~J. Sheu.
\newblock Diffusion for global optimization in {$\mathbb R^n$}.
\newblock {\em SIAM Journal on Control and Optimization}, 25(3):737--753, 1987.

\bibitem[CST20]{chada2020tikhonov}
N.~K. Chada, A.~M. Stuart, and X.~T. Tong.
\newblock Tikhonov regularization within ensemble {K}alman inversion.
\newblock {\em SIAM Journal on Numerical Analysis}, 58(2):1263--1294, 2020.

\bibitem[CV21]{carrillo_vaes_2021_cov_pre_fkp}
J.~A. Carrillo and U.~Vaes.
\newblock Wasserstein stability estimates for covariance-preconditioned
  {F}okker--{P}lanck equations.
\newblock {\em Nonlinearity}, 34(4):2275, 2021.

\bibitem[DL21a]{ding2021ensemble_inversion}
Z.~Ding and Q.~Li.
\newblock Ensemble {K}alman inversion: mean-field limit and convergence
  analysis.
\newblock {\em Statistics and computing}, 31:1--21, 2021.

\bibitem[DL21b]{ding2021ensemble_sampler}
Z.~Ding and Q.~Li.
\newblock Ensemble {K}alman sampler: mean-field limit and convergence analysis.
\newblock {\em SIAM Journal on Mathematical Analysis}, 53(2):1546--1578, 2021.

\bibitem[DMH23]{del2023theoretical_kalman}
P.~Del~Moral and E.~Horton.
\newblock A theoretical analysis of one-dimensional discrete generation
  ensemble {K}alman particle filters.
\newblock {\em The Annals of Applied Probability}, 33(2):1327--1372, 2023.

\bibitem[DMT18]{del_tugaut2018stability}
P.~Del~Moral and J.~Tugaut.
\newblock On the stability and the uniform propagation of chaos properties of
  ensemble {K}alman--{B}ucy filters.
\newblock {\em The Annals of Applied Probability}, 28(2):790--850, 2018.

\bibitem[Eve94]{evensen1994sequential_kalman_filter}
G.~Evensen.
\newblock Sequential data assimilation with a nonlinear quasi-geostrophic model
  using monte carlo methods to forecast error statistics.
\newblock {\em Journal of Geophysical Research: Oceans}, 99(C5):10143--10162,
  1994.

\bibitem[EVL96]{evensen1996assimilation}
G.~Evensen and P.~J. Van~Leeuwen.
\newblock Assimilation of geosat altimeter data for the {A}gulhas current using
  the ensemble {K}alman filter with a quasigeostrophic model.
\newblock {\em Monthly weather review}, 124(1):85--96, 1996.

\bibitem[GIHLS20]{garbunoinigo2019interacting}
A.~Garbuno-Inigo, F.~Hoffmann, W.~Li, and A.~M. Stuart.
\newblock Interacting {L}angevin diffusions: Gradient structure and ensemble
  {K}alman sampler.
\newblock {\em SIAM Journal on Applied Dynamical Systems}, 19(1):412--441,
  2020.

\bibitem[GINR20]{garbuno2020affine}
A.~Garbuno-Inigo, N.~Nusken, and S.~Reich.
\newblock Affine invariant interacting {L}angevin dynamics for {B}ayesian
  inference.
\newblock {\em SIAM Journal on Applied Dynamical Systems}, 19(3):1633--1658,
  2020.

\bibitem[GKM{\etalchar{+}}96]{graham_meleard_1996asymptotic}
C.~Graham, T.~G. Kurtz, S.~M{\'e}l{\'e}ard, P.~E. Protter, M.~Pulvirenti,
  D.~Talay, and S.~M{\'e}l{\'e}ard.
\newblock Asymptotic behaviour of some interacting particle systems;
  {M}ckean-{V}lasov and {B}oltzmann models.
\newblock {\em Probabilistic Models for Nonlinear Partial Differential
  Equations: Lectures given at the 1st Session of the Centro Internazionale
  Matematico Estivo (CIME) held in Montecatini Terme, Italy, May 22--30, 1995},
  pages 42--95, 1996.

\bibitem[HJK11]{hutzenthaler2011strong}
M.~Hutzenthaler, A.~Jentzen, and P.~E. Kloeden.
\newblock Strong and weak divergence in finite time of {E}uler's method for
  stochastic differential equations with non-globally lipschitz continuous
  coefficients.
\newblock {\em Proceedings of the Royal Society A: Mathematical, Physical and
  Engineering Sciences}, 467(2130):1563--1576, 2011.

\bibitem[HJK12]{hutzenthaler_tamed_2012}
M.~Hutzenthaler, A.~Jentzen, and P.~E. Kloeden.
\newblock {Strong convergence of an explicit numerical method for {SDE}s with
  nonglobally {L}ipschitz continuous coefficients}.
\newblock {\em The Annals of Applied Probability}, 22(4):1611 -- 1641, 2012.

\bibitem[HM01]{houtekamer2001sequential}
P.~L. Houtekamer and H.~L. Mitchell.
\newblock A sequential ensemble {K}alman filter for atmospheric data
  assimilation.
\newblock {\em Monthly weather review}, 129(1):123--137, 2001.

\bibitem[HRS24]{hasenpflug2024wasserstein}
M.~Hasenpflug, D.~Rudolf, and B.~Sprungk.
\newblock Wasserstein convergence rates of increasingly concentrating
  probability measures.
\newblock {\em The Annals of Applied Probability}, 34(3):3320--3347, 2024.

\bibitem[HST25]{hinds2024well}
P.~D. Hinds, A.~Sharma, and M.~V. Tretyakov.
\newblock Well-posedness and approximation of reflected {M}c{K}ean-{V}lasov
  {SDE}s with applications.
\newblock {\em Mathematical Models and Methods in Applied Sciences}, 2025.

\bibitem[Hwa80]{hwang1980laplace}
C.-R. Hwang.
\newblock Laplace's method revisited: weak convergence of probability measures.
\newblock {\em The Annals of Probability}, 8(6):1177--1182, 1980.

\bibitem[ILS13]{iglesias2013ensemble}
M.~A. Iglesias, K.~J.~H. Law, and A.~M. Stuart.
\newblock Ensemble {K}alman methods for inverse problems.
\newblock {\em Inverse Problems}, 29(4):045001, 2013.

\bibitem[KS19]{Kovachki_stuart_2019}
N.~Kovachki and A.~Stuart.
\newblock Ensemble {K}alman inversion: a derivative-free technique for machine
  learning tasks.
\newblock {\em Inverse Problems}, 35(9):095005, 8 2019.

\bibitem[KST23]{tretyakov_jumpcbo_2023}
D.~Kalise, A.~Sharma, and M.~V. Tretyakov.
\newblock Consensus-based optimization via jump-diffusion stochastic
  differential equations.
\newblock {\em Mathematical Models and Methods in Applied Sciences},
  33(02):289–339, February 2023.

\bibitem[LMW18]{leimkuhler2018ensemble}
B.~Leimkuhler, C.~Matthews, and J.~Weare.
\newblock Ensemble preconditioning for {M}arkov chain {M}onte {C}arlo
  simulation.
\newblock {\em Statistics and Computing}, 28:277--290, 2018.

\bibitem[LWZ22]{lindsey_weare2022ensemble_teleporting}
M.~Lindsey, J.~Weare, and A.~Zhang.
\newblock Ensemble {M}arkov chain {M}onte {C}arlo with teleporting walkers.
\newblock {\em SIAM/ASA Journal on Uncertainty Quantification}, 10(3):860--885,
  2022.

\bibitem[MPS98]{milstein1998balanced}
G.~N. Milstein, E.~Platen, and H.~Schurz.
\newblock Balanced implicit methods for stiff stochastic systems.
\newblock {\em SIAM Journal on Numerical Analysis}, 35(3):1010--1019, 1998.

\bibitem[MRSS25]{molin2025controlled}
V.~Molin, A.~Ringh, M.~Schauer, and A.~Sharma.
\newblock Controlled stochastic processes for simulated annealing.
\newblock {\em arXiv:2504.08506}, 2025.

\bibitem[MT04]{milstein2004stochastic}
G.~N. Milstein and M.~V. Tretyakov.
\newblock {\em Stochastic numerics for mathematical physics}, volume~39.
\newblock Springer, 2004.

\bibitem[RDF78]{rossky1978brownian}
P.~J. Rossky, J.~D. Doll, and H.~L. Friedman.
\newblock Brownian dynamics as smart {M}onte {C}arlo simulation.
\newblock {\em The Journal of Chemical Physics}, 69(10):4628--4633, 1978.

\bibitem[SS17]{schillings2017analysis_enkf}
C.~Schillings and A.~M. Stuart.
\newblock Analysis of the ensemble {K}alman filter for inverse problems.
\newblock {\em SIAM Journal on Numerical Analysis}, 55(3):1264--1290, 2017.

\bibitem[Szn91]{sznitman1991topics}
A.-S. Sznitman.
\newblock Topics in propagation of chaos.
\newblock {\em Ecole d’{\'e}t{\'e} de probabilit{\'e}s de Saint-Flour
  XIX—1989}, 1464:165--251, 1991.

\bibitem[TZ13]{tretyakov2013fundamental}
M.~V. Tretyakov and Z.~Zhang.
\newblock A fundamental mean-square convergence theorem for sdes with locally
  lipschitz coefficients and its applications.
\newblock {\em SIAM Journal on Numerical Analysis}, 51(6):3135--3162, 2013.

\bibitem[Vae24]{vaes2024sharpchaos}
U.~Vaes.
\newblock Sharp propagation of chaos for the ensemble {L}angevin sampler.
\newblock {\em Journal of the London Mathematical Society}, 110(5):e13008,
  2024.

\bibitem[VdV98]{van2000asymptotic}
A.~W. Van~der Vaart.
\newblock {\em Asymptotic statistics}.
\newblock Cambridge University Press, Cambrdige, UK, 1998.

\bibitem[Vil03]{villani2003topics}
C.~Villani.
\newblock {\em Topics in Optimal Transportation}.
\newblock Graduate studies in mathematics. American Mathematical Society,
  Providence, RI, USA, 2003.

\end{thebibliography}
\bibliographystyle{alpha}
\end{document}